\documentclass[a4paper, 12pt]{amsart} \usepackage[utf8]{inputenc} \usepackage{mathtools} \usepackage[T2A]{fontenc} \usepackage{hyperref} \usepackage{dsfont} \usepackage[T2A]{fontenc} \usepackage{amsmath,amssymb,amsthm} \usepackage[a4paper, lmargin = 3cm, rmargin = 3cm, tmargin = 3cm, bmargin = 2cm]
{geometry} \usepackage[russian, english]{babel} 
\usepackage{xcolor}
\usepackage{tikz-cd} 
\usepackage{cleveref}
\usepackage{multicol}
\usepackage{subcaption}
\crefformat{footnote}{#2\footnotemark[#1]#3}

\usepackage{relsize}

\hypersetup{
    colorlinks,
    linkcolor={blue},
    citecolor={blue},
    urlcolor={blue}
}

\DeclareMathOperator{\E}{\mathbb{E}}

\let\Prob\relax
\DeclareMathOperator{\Prob}{\mathbb{P}}
\DeclareMathOperator{\Cov}{\mathrm{Cov}}

\newtheorem{lemma}{Lemma}
\numberwithin{lemma}{section}
\newtheorem{theorem}{Theorem}
\newtheorem{hypo}{Conjecture}
\newtheorem{prop}{Proposition}

\newtheorem{cor}{Corollary}
\newtheorem{rem}{Remark}

\newtheorem*{hypo*}{Conjecture}

\title{Positivity of partial sums of a random multiplicative function and corresponding problems for the Legendre symbol}
\author{Petr Kucheriaviy}
\address{
Steklov Mathematical Institute of Russian Academy of Sciences, 8 Gubkina st., Moscow, 119991, Russia}
\email{peter.ktchr@gmail.com}

\begin{document}

\begin{abstract}
Let $f(n)$ be a random completely multiplicative function such that $f(p) = \pm 1$ with probabilities $1/2$ independently at each prime. We study the conditional probability, given that $f(p) = 1$ for all $p < y$, that all partial sums of $f(n)$ up to $x$ are nonnegative. We prove that for $y \ge C \frac{(\log x)^2 \log_2 x}{\log_3 x}$ this probability equals $1 - o(1)$.
We also study the probability $P_x$ that $\sum_{n \le x} \frac{f(n)}{n}$ is negative. We prove that $P_x \ll \exp \left( - \exp \left( \frac{\log x \log_4 x}{(1 + o(1)) \log_3 x} \right) \right)$, which improves a bound given by Kerr and Klurman. Under a  conjecture closely related to Halász's theorem, we prove that
$P_x \ll \exp(-x^{\alpha})$ for some $\alpha > 0$.
Let $\chi_p(n) = \left( \frac{n}{p} \right)$ be the Legendre symbol modulo $p$. For a prime $p$ chosen uniformly at random from $(x, 2x]$, we express the probability that all partial sums of $\frac{\chi_p(n)}{n}$ are nonnegative in terms of the probability that partial sums of $\frac{f(n)}{n}$ are nonnegative. 
\end{abstract}

\maketitle

\section{Introduction}

\subsection{Partial sums of $f(n)$}

Let $\chi_p(n) = \left( \frac{n}{p} \right)$ be the Legendre symbol $\pmod p$. 
Let $\mathcal{L}^+$ denote the set of primes $p$ such that the partial sums of $\chi_p(n)$ are all nonnegative. 
The motivating problem for us is whether $\mathcal{L}^+$ is infinite. 
Primes in $\mathcal{L}^+$ are also remarkable, because the corresponding Fekete polynomial has no zeros in $(0, 1)$ 
and it turns out that the
 $L$-function $L(s, \chi_p)$ has no zeros in $(0, 1)$. In particular such $L$-function has no Siegel zeros.

Let $x$ be a large number and let us choose a prime $p \in (x, 2x]$ uniformly at random. 
Kalmynin \cite{kalmynin2023quadratic} proved that $\Prob(p \in \mathcal{L}^+) \ll (\log \log x)^{-c}$, where $c \approx 0.0368$. As indicated in \cite[p. 5]{klurman2024sign}, this bound can be  substantially improved conditionally on the non-existence of Siegel zeros.

Let us define a random completely multiplicative function $f$ by taking $f(p) = \pm 1$ with probabilities $1/2$ independently at each prime. Denote by $\mathcal{F}$ the probability space of such functions.

For a fixed $N$ the random vector $(\chi_p(1), \chi_p(2), \ldots, \chi_p(N))$ converges in distribution to 
$(f(1), f(2), \ldots, f(N))$ as $x \to \infty$ due to  the quadratic reciprocity law and equidistribution of primes in arithmetic progressions (Dirichlet's theorem).
In this sense such a random function $f : [1, N] \cap \mathbb{N} \to \{1, -1\}$ is a good model for $\chi_p : [1, N] \cap \mathbb{N} \to \{1, -1\}$ if $N$ is small enough in comparison with $x$. 

For large $N$ the model becomes inadequate; for instance $\sum_{n = 1}^{p-1} \chi_p(n) = 0$, a property not generally satisfied by an arbitrary $f \in \mathcal{F}$. See \cite{hussain2023limiting}, where a better model on a large scale is defined.

One can try to find $p \in \mathcal{L}^+$ among those primes for which the least quadratic non-residue $n_p$ is large. Then one can expect that $\sum_{n \le y} \chi_p(n)$ is dominated by the contribution of the $(n_p - 1)$-smooth part.
With that in mind, we formulate the following problem. Let us denote by $\mathcal{L}_x^{+}$ the set of completely multiplicative functions $f$ taking values $\pm 1$ such that all partial sums of $f(n)$ up to $x$ are nonnegative. 
What should $y = y(x)$ be so that $
\Prob\left(f \in \mathcal{L}_x^+ \mid f(p) = 1 \, (p \le y)\right) = 1 - o(1)?
$

 \begin{theorem}\label{thm sum f(n)}
There exist $C > 0, x_0 > 0$ such that for any $x > x_0$ and any
\[
y \ge C \frac{(\log x)^2 \log_2 x}{\log_3 x}
\]
we have $\Prob\left(f \in \mathcal{L}_x^+ \mid f(p) = 1 \, (p \le y)\right) = 1 - o(1)$ as $x$ tends to infinity. 
\end{theorem}

It is worth mentioning that the best known lower bound $n_p = \Omega(\log p \log_3 p)$ proved by Graham and Ringrose \cite{graham1990lower} is much smaller than $y$ in Theorem \ref{thm sum f(n)} if we set $p = x$. 
In the setting of Theorem \ref{thm sum f(n)} the value $y = (\log x)^{2 + o(1)}$ seems crucial, because then the square root of the variance of $\sum_{n \le x} f(n)$ surpasses the expectation  $\Psi(x, y) = x^{1/2 + o(1)}$. Let us state this as a conjecture.

\begin{hypo} \label{conj1}
For any $\varepsilon > 0$, $x > x_0(\varepsilon)$ and $y \le (\log x)^{2 - \varepsilon}$ we have
\begin{equation} \label{L_x^+ condition}
\Prob\left(f \in \mathcal{L}_x^+ \mid f(p) = 1 \, (p \le y)\right) = o(1).
\end{equation}
\end{hypo}

Recently Conjecture \ref{conj1} was proved by Angelo and Xu \cite{angelo2026oscillations}. 
Moreover, they proved that (\ref{L_x^+ condition}) holds
for $y = o \left( \left( \frac{\log x}{\log_2 x} \right)^2 \right)$. Hence the bound on $y$ in Theorem \ref{thm sum f(n)} is tight up to a power of $\log_2 x$.

Since 
\[
\Prob\left(f \in \mathcal{L}_x^+\right) \ge \Prob(f(p) = 1 \, (p \le y))
\Prob\left(f \in \mathcal{L}_x^+ \mid f(p) = 1 \, (p \le y)\right) = 2^{-\pi(y)} (1 - o(1)),
\]
Theorem \ref{thm sum f(n)} gives us the following corollary. 
\begin{cor} \label{immediate cor}
\[
\Prob\left(f \in \mathcal{L}_x^+\right)  \ge \exp \left( - C' \frac{(\log x)^2}{\log_3 x} \right).
\]
\end{cor}
The upper bound $\Prob \left(f \in \mathcal{L}_x^+\right) \ll (\log x)^{-c + o(1)}$ was proved in \cite{kalmynin2023quadratic}. It is plausible that  Corollary \ref{immediate cor} can be substantially improved, since we used a very special construction to detect $f$ in $\mathcal{L}_x^{+}$.

\subsection{Partial sums of $\frac{f(n)}{n}$}

Now let us ask:
What is the probability that the  sums $\sum_{n \le y} \frac{\chi_p(n)}{n}$ are positive for all $y \ge 1$? This problem seems to be much more approachable than the problem about $\mathcal{L}^+$. First of all, $\sum_{n} \frac{\chi_p(n)}{n}$ converges to $L(1, \chi_p)$, which is positive due to Dirichlet’s class number formula. This shows that the partial sums $\sum_{n \le y} \frac{\chi_p(n)}{n}$ are all strictly positive from a certain point. 
Second, the values of $\chi_p(q)$ for large primes $q$ have an insignificant influence on the size of $\sum_{n \le y} \frac{\chi_p(n)}{n}$. Hence, it is easier to prove that a random function $f \in \mathcal{F}$ is a good model for $\chi_p$ in this problem.

Let us start by discussing the analogous problem for $f \in \mathcal{F}$. First, what can be said about an arbitrary fixed $f \in \mathcal{F}$?
The question of how negative the sum $\sum_{n \le x} \frac{f(n)}{n}$ can be was discussed by Granville and Soundararajan \cite{granvillenegative}. They showed among other things that for $x$ sufficiently large $\sum_{n \le x} \frac{f(n)}{n} \ge -(\log \log x)^{-3/5}$ and constructed $f$ such that $\sum_{n \le x} \frac{f(n)}{n} < - \frac{c}{\log x}$.
Kerr and Klurman \cite{kerr2022negative} proved that $\sum_{n \le x} \frac{f(n)}{n} \ge -(\log \log x)^{-1 + o(1)}$. Recently it was proved by Klurman and Mangerel \cite{KM} that $\sum_{n \le x} \frac{f(n)}{n} \ge -\frac{c}{(\log x)^{1 - 2/\pi}}$.
Of course these results can be applied to $f(n) = \chi_p(n)$.

Now let us denote by $P$ the probability that 
\[
\sum_{n \le y} \frac{f(n)}{n} > 0
\]
for every $y \ge 1$. Angelo and Xu \cite{angelo2023turan} proved that $1 - 10^{-45} < P$.
Let $\lambda(n)$ be the Liouville function.
Borwein, Ferguson,
and Mossinghoff \cite{borwein2008sign} showed that the minimal $N_0$ such that
\[
\sum_{n \le N_0} \frac{\lambda(n)}{n} < 0
\]
is $N_0 = 72, 185, 376, 951, 205$. Hence
\[
P \le 1 - 2^{-\pi(N_0)} < 1 - 10^{-704 \times 10^9}.
\]

Let us denote by $P_x$ the probability that
\[
\sum_{n \le x} \frac{f(n)}{n} < 0.
\]
It turns out that $P_x$ tends to $0$ very rapidly.  

Angelo and Xu proved \cite[Theorem 1.2]{angelo2023turan} that 
\[
P_x \ll \exp \left( - \exp \left( \frac{\log x}{C \log_2 x} \right) \right).
\]

Kerr and Klurman \cite[Theorem 1.2]{kerr2022negative} improved this to
\begin{equation} \label{Kerr-Klurman bound2}
P_x \ll \exp \left( - \exp \left( \frac{\log x \log_3 x}{C \log_2 x} \right) \right),
\end{equation}
for some constant $C$. Although the authors do not state it explicitly, one can derive from the proof that $C = 1 + o(1)$ is admissible.

\begin{theorem} \label{Kerr and Klurman improvement}
\[
P_x \ll \exp \left( - \exp \left( \frac{\log x \, \log_4 x}{(1 + o(1)) \log_3 x} \right) \right),
\]
as $x \to +\infty$.
\end{theorem}

This theorem can be improved if the following conjecture is true.

\begin{hypo} \label{smooth Halasz conj}
There exists $\varepsilon > 0$ such that for all $x \ge 3$ and all real-valued completely multiplicative functions $f$ such that $|f(n)| \le 1$ for all $n$ and $f(p) = 0$ for all $p > x^{\varepsilon}$ we have
\[
\sum_{n \le x} f(n) \ll \frac{x}{\log \log x} \exp \left( \sum_{p \le x} \frac{f(p)}{p} \right).
\]
\end{hypo}

\begin{theorem} \label{cor from conj}
    If Conjecture \ref{smooth Halasz conj} is true, then there exists $\alpha > 0$ such that 
    \[
    P_x \ll \exp(-x^{\alpha}).
    \]
\end{theorem}

Recently, the bound $P_x \ll \exp(-x^{\alpha})$ was established unconditionally by Klurman and Mangerel \cite[Theorem 2.3]{KM}, by a different method (by using a Halász-type asymptotic formula for logarithmic means). Their argument does not settle Conjecture \ref{smooth Halasz conj}, which remains open.

We now comment on Conjecture \ref{smooth Halasz conj}. Proposition \ref{sum f(n) upper bound in terms of sum f(p)/p} is a weaker result that we use instead of Conjecture \ref{smooth Halasz conj} to prove Theorem \ref{Kerr and Klurman improvement}. To prove Proposition \ref{sum f(n) upper bound in terms of sum f(p)/p}, we use a version of Halász's theorem for multiplicative functions with support on smooth numbers, which was proved by Granville, Harper, and Soundararajan in \cite{granville2018more} (see Lemma \ref{Halász}).
If we take $f(p) = -1$ for $p < x/2$, and $f(p) = 0$ otherwise, then 
\[
\sum_{n \le x} f(n) \asymp \frac{x}{\log x} \asymp x \exp\left( \sum_{p \le x} \frac{f(p)}{p} \right).
\]
Hence the condition that $f$ is supported on $x^{\varepsilon}$-smooth numbers cannot be dropped. 

We have no strong evidence for Conjecture \ref{smooth Halasz conj} and would not be surprised by a counterexample. We pose it because it implies $P_x \ll \exp(-x^{\alpha})$ (now also established unconditionally \cite{KM}) and because its failure would show that Halász's theorem is tight at this scale.

An application of Halász's theorem similar to the one in the proof of Proposition \ref{sum f(n) upper bound in terms of sum f(p)/p} implies that a counterexample $f = f_x$, if it exists, must satisfy $\sum_{p \le x} \frac{f(p)}{p} < (-1 + \delta) \log \log x$ for every $\delta > 0$ and all sufficiently large $x$. 

The naive attempt to construct a counterexample is to let $f(p) = -1$ for $p \le x^{0.99}$ and $f(p) = 0$ otherwise. But then, as follows from the result of Alladi
\cite[Theorem 2]{alladi1982asymptotic},
\[
\sum_{n \le x} f(n) \ll \frac{x}{(\log x)^2}.
\]
 Therefore, this does not produce a counterexample to Conjecture \ref{smooth Halasz conj}.

Now let us return to the problem about $\chi_p$.
Denote by $\tilde{P}_x$ the probability that for all $y \ge 1$
\[
\sum_{n \le y} \frac{\chi_p(n)}{n} > 0.
\]

\begin{theorem} \label{chi(n)/n}
    Let $p$ be a random prime in $(x, 2x]$ chosen uniformly.
    Let $A$ be the event in $\mathcal{F}$ that partial sums of the sequence $\frac{f(n)}{n}$ are all positive. Let $k = 8 \prod_{2 < q \le c_1 \sqrt{\log x}} q$. Let $E_0 = 0$ if no character $\pmod{k}$ has a Siegel zero. If there is a character $\chi_1 \pmod{k}$ with Siegel zero $\beta_1$, then $\chi_1$ can be written in the form $\chi_1(n) = \left( \frac{d}{n} \right)$, where $d | k$. In this case, we set $E_0 = 0$ if $d < 0$ and $E_0 = 1$ otherwise.
    
    Then
    \[
    \tilde{P}_x = P - E_0 \Cov \left( \mathds{1}_{A}, f(d) \right) \frac{\int_x^{2x} \frac{u^{\beta_1 - 1}}{\log u} \, du}{\operatorname{Li}(2x) - \operatorname{Li}(x)} + O \left(  \exp \left( - \exp \left( \frac{\log_2 x \log_4 x}{(2 + o(1)) \log_3 x} \right) \right) \right).
    \]
\end{theorem}

\begin{cor}\label{Cor chi(n)/n}
\[
\tilde{P}_x - P \ll \exp \left( - \exp \left( \frac{\log_3 x \log_6 x}{(1 + o(1)) \log_5 x} \right) \right).
\]
\end{cor}

\begin{rem}
By using the bound $P_x \ll \exp(-x^{\alpha})$ proved in \cite[Theorem 2.3]{KM} instead of Theorem \ref{Kerr and Klurman improvement}, one can sharpen Corollary \ref{Cor chi(n)/n} and obtain
\[
\tilde{P}_x - P \ll \exp(- (\log_2 x)^{\alpha}).
\]
\end{rem}

We studied in this paper the positivity of partial sums $\sum_{n \le y} f(n)$ and $\sum_{n \le y} \frac{f(n)}{n}$. The analogous problems for  $\sum_{n \le y} \frac{f(n)}{n^{\sigma}}$, where $0 < \sigma < 1$, were investigated in recent years. See, for example, \cite{aymone2025positivity}, \cite{aymone2024sign}; the latter includes a survey of results on random multiplicative functions.

\subsection{Notation}

Throughout, $\log$ denotes the natural logarithm, and for an integer $k \ge 1$ we write
\[
  \log_k x = \underbrace{\log \log \cdots \log}_{k} x
\]
for the $k$-fold iterated logarithm. We write $\operatorname{Li}(x) = \int_{2}^{x} \frac{dt}{\log t}$ for the logarithmic integral.
We write $\rho$ for the \emph{Dickman function}, i.e.\ the unique continuous function
on $[0, \infty)$ with $\rho(u) = 1$ for $0 \le u \le 1$ and
$u \rho'(u) + \rho(u - 1) = 0$ for $u > 1$.

For a nonnegative
function $g$, we write $f = O(g)$ or $f \ll g$ if $|f| \le C g$ for some constant
$C > 0$. We write
$f \asymp g$ when $f \ll g$ and $f \gg g$ hold simultaneously, $f \sim g$ when
$f/g \to 1$, and $f = o(g)$ when $f/g \to 0$. Implied constants are absolute unless
indicated otherwise. A subscript such as $O_{\varepsilon}(\cdot)$ or $\ll_{\varepsilon}$ means that the implied
constant may depend on the subscripted parameter.
We write $n = \square$ to mean that $n$ is a perfect square; thus a sum such as
$\sum_{ab = \square}$ is taken over all pairs with $ab$ a perfect square.
A flat on a summation sign, $\sum^{\flat}$, indicates that the sum is restricted to
square-free integers.

For random variables $X$ and $Y$, we write $\Prob(\cdot)$ and $\E(\cdot)$ for
probability and expectation, $\Cov(X, Y)$ for the covariance of $X$ and $Y$, and
$\mathds{1}_{A}$ for the indicator function of an event $A$.

$\left(\frac{m}{n} \right)$ denotes the Kronecker symbol.

For a positive integer $n$, let $P(n)$ and $p(n)$ denote its largest and smallest
prime divisors, with the conventions $P(1) = 1$ and $p(1) = +\infty$. For $x \ge 1$
and $y > 0$ we set
\[
  S(x, y) := \{ n \le x : P(n) \le y \}, \qquad
  R(x, y) := \{ n \le x : p(n) > y \},
\]
the $y$-smooth and $y$-rough integers up to $x$, and we let $\Psi(x, y) := |S(x, y)|$
denote the number of $y$-smooth integers up to $x$. Note that $1 \in R(x, y)$ for
every $x \ge 1$, and that $R(x, y) = \{ 1 \}$ whenever $y \ge x$. For a positive integer $n$, let $\omega(n)$ denote the number of distinct prime divisors of $n$.

Throughout, $c$ and $C$ denote positive absolute constants whose values may change
from one occurrence to the next. We attach subscripts $c_1, c_2, \dots$ and
$C_1, C_2, \dots$ to label specific constants where convenient.

\section{Outline of the proofs}

\subsection{Theorem \ref{thm sum f(n)}}
Let $f$ be a random completely multiplicative function taking values $\pm 1$ such that $f(p) = 1$ for all $p \le y$.

First, for an individual $z \in [y, x]$, we bound the probability that $\sum_{n \le z} f(n)$ is negative (Lemma \ref{sum f(n) bound lemma1}). 
The equality (\ref{sum f(n) Psi* eq}) provides a positive main part of $\sum_{n \le z} f(n)$ and the remainder random part, which we want to show to be small with high probability. The main part and the coefficients of the random part in (\ref{sum f(n) Psi* eq}) are written in terms of the $\Psi(t, y)$ function.

There are two key ingredients in the proof of Lemma \ref{sum f(n) bound lemma1}. 
The first ingredient is the de la Bret{\`e}che--Tenenbaum theorem (Lemma \ref{alpha ineq}), which allows us to compare $\Psi(z, y)$ with $\Psi(z/t, y)$. We note that the analogous result of Hildebrand and Tenenbaum \cite[Theorem 3]{hildebrand1986integers} requires the additional restriction $t \le y$, whereas Lemma \ref{alpha ineq} is uniform in $t$.
The second ingredient is the
Bonami--Halász inequality (Lemma \ref{Halasz lemma}), which  provides upper bounds on the moments of the remainder random part of $\sum_{n \le z} f(n)$.
Applying Markov's inequality $\Prob(|X| \ge a) \le a^{-q}\E[X^q]$ (with $q$ even), we bound the probability that the remainder term is large.

Having controlled $\sum_{n \le z} f(n)$ for a fixed $z$, we choose checkpoints $z_0 < z_1 < \ldots < z_k = x$ and apply Lemma \ref{sum f(n) bound lemma1} at each. We want $z_i$ to be as widely spaced as possible, while still controlling $\sum f(n)$ over each short interval $[z_i, z_{i+1}]$.

The trivial bound $\left|\sum_{z_{i} \le n \le u} f(n) \right| \le u - z_i + 1$ allows us to take $z_{i + 1} > z_{i} + z_i^{1/2 + o(1)}$. This gives the conclusion of Theorem \ref{thm sum f(n)} for the weaker range $y \ge (\log x)^{3 + o(1)}$. 
The same range, up to a constant multiple, is obtained by taking $z_{i + 1} = z_{i} + 1$, and we do this in Lemma \ref{first bound y}, proving that, with probability $1 - o(1)$, $\sum_{n \le z} f(n) \ge 0$ for all $z \le X$, where $(\log X)^3 \approx y$. 

To reach the stated range $y \ge (\log x)^{2 + o(1)}$, we need to take the checkpoints more sparsely. We take $x_0 = X$, $x_{i + 1} \approx x_i + \frac{x_i}{(\log x_i)^3}$.
We note that even if we could take $x_{i+1} = 2x_i$ and control the sum of $f(n)$ between these checkpoints, Lemma \ref{sum f(n) bound lemma1} would give us the same conclusion as Theorem \ref{thm sum f(n)}, only with a better constant $C$.

To control the sum of $f(n)$ on the short interval $[x_i, x_{i+1}]$, it suffices to control the sums over dyadic blocks
$(x_i + 2^k l, x_i + 2^k(l + 1)]$, for which we use the identity (\ref{f(n) 2^k interval sum}), analogous to (\ref{sum f(n) Psi* eq}).
The coefficients of $f(n)$ now count the $y$-smooth integers in unions of short intervals. To bound them, we use a triangle inequality for smooth numbers due to Hildebrand (Lemma \ref{triangle smooth}), namely $\Psi(x + z, y) - \Psi(x, y) \le \Psi(z, y) + 1$. 
Applying it to each interval bounds every coefficient by a sum of terms $\Psi(z, y)$ together with a sum of $1$'s. Accordingly we split the whole sum into two — one carrying the $\Psi(z, y)$ contributions, the other the $1$'s (see \eqref{two sequences div}) — and apply Markov's inequality and Bonami--Halász inequality, with moments of different orders in the two cases.

\subsection{Theorem \ref{Kerr and Klurman improvement}}

To provide an upper bound on $P_x$ Angelo and Xu \cite[Theorem 1.2]{angelo2023turan} 
used the identity
\begin{equation} \label{angelo_xu_identity}
    \sum_{n \le x} \frac{f(n)}{n} = \prod_{p \le x} \left( 1 - \frac{f(p)}{p} \right)^{-1} - \sum_{\substack{n > x \\ P(n) \le x}} \frac{f(n)}{n}.
\end{equation}
The first term on the right-hand side is obviously positive and can be proved to be relatively large with high probability. After that, one provides an upper bound on the absolute value of the second term which holds with high probability.

Let $g = f * 1$. Then $g$ is multiplicative and nonnegative. The following equation is the analog of (\ref{angelo_xu_identity}) and is the basis for the proof of Theorem \ref{Kerr and Klurman improvement} and Proposition \ref{high moment sum f(n) ineq}.
\begin{equation} \label{main elementary eq}
    \sum_{n \le x} \frac{f(n)}{n} = \frac{1}{x} \sum_{n \le x} g(n) + \frac{1}{x} \sum_{n \le x} f(n) \left\{ \frac{x}{n} \right\}.
\end{equation}

The proof of \cite[Theorem 1.2]{kerr2022negative} by Kerr and Klurman goes as follows. The authors show that the two probabilities 
\begin{equation}
P_{0} = \Prob\left( \sum_{n \le x} g(n) < \frac{cx}{\log x} \right), \qquad
P_1 = \Prob\left( \left|\sum_{n \le x} f(n) \left\{ \frac{x}{n} \right\} \right| \ge \frac{cx}{\log x} \right)
\end{equation}
are small. 

Note that $\sum_{n \le x} g(n) \ge \sum_{p \le x} g(p) = \sum_{p \le x} (1 + f(p))$ and the good upper bound on $P_0$ comes from a Chernoff-type bound.
To bound $P_1$ the authors use \cite[Proposition 5.6]{kerr2022negative}: a high-order moment inequality, with moment order as large as $\exp \left( \frac{\log x \log_3 x}{C \log_2 x} \right)$.

Let us state \cite[Proposition 5.6]{kerr2022negative} in a slightly generalized form.

\begin{prop} \label{high moment sum f(n) ineq}
Let $\beta_0 > 0$ and $\beta(x) = \beta_0 + o(1)$ as $x \to \infty$. Let $q$ be an even positive integer. Then there exists a function $o_{\beta}(1)$ such that
\[
\E \left[ \left( \sum_{n \le x} f(n) \right)^q \right]^{1/q} = o \left( \frac{x}{(\log x)^{\beta(x)}} \right)
\]
uniformly for
\[
q \le \exp \left( \frac{\log x \log_3 x}{(\beta_0 + o_{\beta}(1)) \log_2 x} \right).
\]
\end{prop}

We can hope to improve the result of \cite[Theorem 1.2]{kerr2022negative} by providing a better lower bound on $\sum_{n \le x} g(n)$, which still holds with high enough probability. If we use Proposition \ref{high moment sum f(n) ineq}, then we need a good upper bound on
\[
P_{\varepsilon} := \Prob \left( \sum_{n \le x} g(n) < \frac{x}{(\log x)^{1 - \varepsilon}} \right)
\]
for some $\varepsilon > 0$ to improve the constant $C = 1$ in (\ref{Kerr-Klurman bound2}).
Unfortunately, this approach does not work, as one can show that

\[
P_{\varepsilon} \gg \exp \left( - \exp \left( (\log x)^{1-\varepsilon} \right) \right).
\]

Let us note, however, that in view of (\ref{main elementary eq}) we only need
\begin{equation} \label{g > f}
\sum_{n \le x} g(n) \ge \left| \sum_{n \le x} f(n) \left\{ \frac{x}{n} \right\} \right|,
\end{equation}
to guarantee that $\sum_{n \le x} \frac{f(n)}{n} \ge 0$. The two sides of (\ref{g > f}) are strongly dependent. Exploiting this dependence is one of the main ideas in the proof of Theorem \ref{Kerr and Klurman improvement}.

To make it explicit, we provide a lower bound for $\sum_{n \le x} g(n)$, in terms of $f$, that holds with high probability (Proposition \ref{g(n) final lower bound}). 
Note that this lower bound $\sum_{n \le x} g(n) \gg x \exp\left( \sum_{p \le x} \frac{f(p)}{p} \right)$ is best possible up to a constant factor, since for all $f$ the upper bound 
\[
\sum_{n \le x} g(n) \ll \frac{x}{\log x} \prod_{p \le x} \left( 1 + \frac{g(p)}{p} + \frac{g(p^2)}{p^2} + \ldots \right) \asymp x \exp\left( \sum_{p \le x} \frac{f(p)}{p} \right)
\]
holds (see, for example, \cite{halberstam1979result}).

The problem reduces to showing that the moment $S$ in (\ref{S moment def}) is $o(1)$ with $q$ taken as large as possible. To bound $S$, we follow the proof of \cite[Proposition 5.6]{kerr2022negative}: we split $S \le S_1 + S_2$, where $S_1$ and $S_2$ are the moments of the sums over integers with large rough part and large smooth part, respectively. We bound $S_1$ by Rankin's trick, while bounding $S_2$ reduces to bounding $\exp \left( - \sum_{p \le x} \frac{f(p)}{p} \right) \Psi_f(x, y)$, where $\Psi_f(x, y) = \sum_{n \in S(x, y)} f(n)$. The trivial estimate 
\[
\exp \left( - \sum_{p \le x} \frac{f(p)}{p} \right) \Psi_f(x, y) \ll (\log x) \Psi(x, y)
\]
gives the Kerr and Klurman bound (\ref{Kerr-Klurman bound2}).

It turns out that if a bound of the form $|\Psi_f(x, y)| \le c_f(y) \Psi(x, y)$ holds in a range $y^a \le x \le y^{a+1}$, then it holds in the much larger range $x \ge y^a$, $\log y \ge (\log_2 x)^{5/3 + \varepsilon}$. The precise statement is Proposition \ref{Psi_f proposition}. The proof uses an induction argument similar to the one used by Hildebrand \cite[Theorem 1]{hildebrand1986number} to prove that $\Psi(x, y) \sim x \rho\left( \frac{\log x}{\log y} \right)$ in the range $\log y \ge (\log_2 x)^{5/3 + \varepsilon}$.

Hence we only need a nontrivial bound for $\Psi_f(x, y)$ in terms of $\sum_{p \le x} \frac{f(p)}{p}$ in the range $y^a \le x \le y^{a+1}$, and Proposition \ref{Psi_f proposition} then extends it to the larger range. A natural source of bounds of this type is Halász's theorem, which bounds partial sums of multiplicative functions in terms of the values of the corresponding $L$-function on the line $1 + it$. We use a version due to Granville, Harper, and Soundararajan \cite{granville2018more} (Lemma~\ref{Halász}), which is slightly sharper for sums of $f(n)$ over smooth numbers; this is exactly our setting, since we may take $a > 1$. Proposition~\ref{sum f(n) upper bound in terms of sum f(p)/p} is a corollary of Halász's theorem and bounds $\Psi_f(x, y)$ in terms of $\sum_{p \le x} \frac{f(p)}{p}$. Here it is essential that $f$ take real values.

\subsection{Theorem \ref{chi(n)/n}}
We give an upper bound on the probability that the partial sum $\sum_{n \le y} \frac{\chi_p(n)}{n}$ becomes negative for some $y \ge N$, where $N$ is a large fixed number (Lemma \ref{P_x([N, infty)]) bound}).

The series $\sum_{n=1}^{\infty} \frac{\chi_p(n)}{n}$ converges to $L(1, \chi_p) > 0$. Using a lower bound for $L(1, \chi_p)$, we show that the partial sum is nonnegative for all sufficiently large $y$. For the remaining $y > N$ we use the identity (\ref{main elementary eq}) to extract the positive main term, and the moment inequality then provides the bound. Here we use the classical second-moment bound of Elliott (Lemma \ref{Second moment ineq}) together with a high-moment bound (Lemma \ref{high moment inequality}), which we deduce from Proposition \ref{high moment sum f(n) ineq}.

Finally, on the short interval $[1, N]$ the distributions of $(\chi_p(1), \ldots, \chi_p(N))$ and $(f(1), \ldots, f(N))$ are close. A Siegel zero, if it exists, shifts the distribution of $(\chi_p(1), \ldots, \chi_p(N))$, and it turns out that the effect of this shift on $\tilde{P}_x$ can be expressed in terms of $\Cov \left( \mathds{1}_{A}, f(d) \right)$.

\section{Proof of Theorem \ref{thm sum f(n)}}

Denote by $\alpha = \alpha(x, y)$ the solution to the equation 
\[
\sum_{p \le y} \frac{\log p}{p^{\alpha} - 1} = \log x.
\]

\begin{lemma}\label{keylemma1}
There exist $K > 0$ and $y_0 > 0$ such that for any $y_0 < y < x$ that satisfies $y = (\log x)^{2 + \varepsilon}, \varepsilon > (\log x)^{-1/4}$ and
\begin{equation} \label{keylemma1equation}
(2 \alpha(x, y) - 1) y^{2 \alpha(x, y) - 1} \log y \ge K \log \log x
\end{equation}
we have 
$\Prob\left(f \in \mathcal{L}_x^+ \mid f(p) = 1 \, (p \le y)\right) = 1 - o(1)$.
\end{lemma}

\begin{proof}[Deduction of Theorem \ref{thm sum f(n)} from Lemma \ref{keylemma1}]
\[
\log x = \sum_{p \le y} \frac{\log p}{p^{\alpha} - 1} \ge \frac{\theta(y)}{y^{\alpha} - 1},
\]
where $\theta(y) = \sum_{p \le y} \log p$.
Hence for $y \gg (\log x)^{1.01}$
\[
\alpha \ge \frac{\log\left(1 + \frac{\theta(y)}{\log x}\right)}{\log y} = 
\frac{\log\left(\frac{y}{\log x}\right)}{\log y} +
o\left( \frac{1}{\log y} \right).
\]
If $y = (\log x)^{2 + \varepsilon}$, then
\[
2 \alpha(x, y) - 1 \ge \frac{\varepsilon}{2 + \varepsilon} + o(1/\log y).
\]
Hence
\begin{equation} \label{gg we need}
(2 \alpha(x, y) - 1) y^{2 \alpha(x, y) - 1} \log y \gg
\varepsilon y^{\frac{\varepsilon}{2 + \varepsilon}} \log y.
\end{equation}
We want the right-hand side to be $\gg \log \log x \asymp \log y$. 
This is satisfied if
\[
\varepsilon = 
\frac{\log_3 x}{\log_2 x} - \frac{\log_4 x}{\log_2 x} + \frac{R}{\log_2 x},
\]
where $R$ is a sufficiently large constant.
Thus we can take
\[
y = (\log x)^{2 + \varepsilon} \gg \frac{(\log x)^2 \log_2 x}{\log_3 x}.
\]

\end{proof}

From now on in this section, $f$ is a random completely multiplicative function taking values $\pm 1$ such that $f(p) = 1$ for all $p \le y$.

Let
\[
\Psi^*(x, y) := \sum_{p(m) > y} \Psi\left( \frac{x}{m^2}, y \right).
\]

Each $n \in \mathbb{N}$ can be uniquely written in the form  $m^2 s n_0$, where $s, n_0$ are square-free and $P(s) \le y, p(n_0) > y$. Since $f(m^2 s n_0) = f(n_0)$, we obtain

\begin{equation} \label{sum f(n) Psi* eq}
\sum_{n \le x} f(n) = \Psi^*(x, y) + \sum_{\substack{p(n) > y \\ n \ne 1}}{}^{\flat} f(n) \Psi^*\left( \frac{x}{n}, y \right).
\end{equation}

Our plan to prove Lemma \ref{keylemma1} is to show that the positive main term $\Psi^*(x, y)$ dominates the remaining sum with high probability.

We will need some auxiliary results.

\subsection{Auxiliary results}

\begin{lemma} \label{alpha ineq}
Uniformly in $1 \le t \le x$, $2 \le y \le x$, we have
\[
\Psi\left(\frac{x}{t}, y \right) \ll \frac{\Psi(x, y)}{t^{\alpha(x, y)}}.
\]
\end{lemma}

\begin{proof}
This is proved by de la Bret{\`e}che and Tenenbaum \cite[Theorem 2.4]{de2005proprietes}, using a formula for $\Psi(x, y)$ due to Hildebrand and Tenenbaum \cite[Theorem 1]{hildebrand1986integers} that was obtained via Perron's formula and the saddle-point method.
\end{proof}

\begin{lemma} \label{alpha}
Uniformly in $x \ge y \ge 2$, we have
\[
\alpha(x, y) = \frac{\log(1 + y / \log x)}{\log y} \left( 1 + O \left( \frac{\log_2 (1 + y)}{\log y} \right) \right).
\]
\end{lemma}

\begin{proof}
See \cite[Theorem 2]{hildebrand1986integers}.
\end{proof}

\begin{lemma}[Bonami--Halász inequality] \label{Halasz lemma}
Let $h(n)$ be a Rademacher random multiplicative function and let $b_j(n) \in \mathbb{C}$ be fixed coefficients. Then
\[
\left| \E \left( \prod_{1 \le j \le m} \sum_{n \ge 1}{}^{\flat} b_j(n) h(n) \right) \right| 
\le
\left( \prod_{1 \le j \le m} \sum_{n \ge 1}{}^{\flat} |b_j(n)|^2  (m - 1)^{\omega(n)} \right)^{1/2}.
\]
\end{lemma}

\begin{proof}
This statement was proved by Bonami in \cite{bonami1970etude}. See \cite[Lemma 2]{halasz1983random} for an alternative proof.
\end{proof}

\begin{lemma} \label{triangle smooth}
There exists a constant $y_1 \ge 2$ such that for all $y \ge y_1$, $x, z > 0$, we have
\[
\Psi(x + z, y) - \Psi(x, y) \le \Psi(z, y) + 1,
\]
and 
\[
\Psi^{*}(x + z, y) - \Psi^*(x, y) \le \Psi^*(z, y) + \sqrt{x + z}.
\]
\end{lemma}
\begin{proof}
    The first inequality was proved by Hildebrand \cite[Theorem 4]{hildebrand1985integers} without $+1$ on the right side but with an additional assumption that $x, z \ge y$. But if $z < y$, then
    \[
    \Psi(x + z, y) - \Psi(x, y) \le [z] + 1 = \Psi(z, y) + 1
    \]
    still holds. The same argument works if $x < y$. This proves the first inequality.

    Now we have
    \begin{multline*}
    \Psi^*(x + z, y) - \Psi^*(x, y) = \sum_{\substack{p(m) > y \\ m \le \sqrt{x + z}}} \left( \Psi\left( \frac{x + z}{m^2}, y \right) - \Psi\left( \frac{x}{m^2}, y \right) \right) \le \\
    \sum_{\substack{p(m) > y \\ m \le \sqrt{x + z}}} \left( \Psi\left( \frac{z}{m^2}, y \right) + 1 \right) \le \Psi^*(z, y) + \sqrt{x + z}.
    \end{multline*}
\end{proof}

Konyagin and Pomerance \cite{konyagin1997primes} showed the following lower bound.
\begin{lemma} \label{Konyagin}
If $x \ge 4$ and $2 \le y \le x$, then
    \[
    \Psi(x, y) \ge x^{1 - \frac{\log_2 x}{\log y}}.
    \]
\end{lemma}
\begin{proof}
    See \cite[Theorem 2.1]{konyagin1997primes}.
\end{proof}

\subsection{Proof of Lemma \ref{keylemma1}}

Let $y = (\log x)^{2 + \varepsilon}$, $\varepsilon > 0$. Then, by Lemma \ref{alpha}
\[
\alpha(x, y) = \frac{1 + \varepsilon}{2 + \varepsilon} \left( 1 + O \left( \frac{\log_2 y}{\log y} \right) \right).
\]
Note that $\frac{1 + \varepsilon}{2 + \varepsilon} = \frac{1}{2} + \frac{\varepsilon}{2(2 + \varepsilon)}$.

By Lemma \ref{alpha ineq}
\begin{equation} \label{Psi* and Psi}
\frac{\Psi^*(x, y) - \Psi(x, y)}{\Psi(x, y)} \ll \sum_{\substack{p(m) > y \\ m \ne 1}} m^{-2 \, \alpha(x, y)} =
\prod_{p > y} \left( 1 - p^{-2\alpha(x, y)} \right)^{-1} - 1.
\end{equation}

We have
\[
\sum_{p > y} p^{-2 \alpha(x, y)} \ll \frac{y^{1 - 2\alpha}}{(\log y)(2 \alpha - 1)}.
\]

Let us assume that $y$ is such that $\frac{y^{1 - 2\alpha}}{(\log y)(2 \alpha - 1)} = o(1)$. Then (\ref{Psi* and Psi}) gives us
\[
\frac{\Psi^*(x, y) - \Psi(x, y)}{\Psi(x, y)} \ll \sum_{p > y} p^{-2 \alpha(x, y)} = o(1).
\]

Hence, under this condition, $\Psi^*(u, y) \sim \Psi(u, y)$ for any $u \le x$ since $\alpha(x, y)$ is monotonically decreasing in $x$.

\begin{lemma} \label{sum f(n) bound lemma1}
Let $\delta > 0$, $x > y$ and
\[
R_x := \Prob\left( \left| \sum_{\substack{p(n) > y \\ n \ne 1}}{}^{\flat} f(n) 
\Psi^* \left( \frac{x}{n}, y \right) \right| > \delta \Psi^*(x, y)
\right).
\]

Suppose that $y = y(x)$ satisfies $\frac{y^{1 - 2\alpha}}{(\log y)(2 \alpha - 1)} = o(1)$.
Then there exists an absolute constant $c_0 > 0$ such that

\[
R_x \ll \exp \left( - c_0 \delta^2 (2 \alpha - 1) y^{2 \alpha - 1} \log y  \right).
\]

\end{lemma}

\begin{proof}

By Lemma \ref{alpha ineq} we have 
\[
\Psi^*\left(\frac{x}{n}, y \right) \le C_1 n^{-\alpha(x, y)} \Psi^*(x, y)
\]
for some absolute constant $C_1$.
Thus Lemma \ref{Halasz lemma} and Markov's inequality give us
\[
R_x \ll C_1^{2m} \delta^{-2m} \left( \sum_{\substack{p(n) > y \\ n \ne 1}}{}^{\flat} n^{-2\alpha(x, y)} (2m - 1)^{\omega(n)} \right)^m.
\]

Hence
\[
R_x \ll C_1^{2m} \delta^{-2m} \left( \prod_{p > y}\left(1 + \frac{2m - 1}{p^{2 \alpha}} \right) - 1 \right)^m.
\]

Suppose that $\sum_{p > y} \frac{m}{p^{2 \alpha}} < 1/2$. Let $C_2 = \max(2 C_1^2, 2)$. Then
\[
R_x \ll C_1^{2m} \delta^{-2m} \left( \exp \left( 2 \sum_{p > y} \frac{m}{p^{2 \alpha}} \right) - 1 \right)^m \ll C_2^m \delta^{-2m}
\left(
\sum_{p > y} \frac{m}{p^{2 \alpha}}
\right)^m.
\]

The bound still holds if $\sum_{p > y} \frac{m}{p^{2 \alpha}} \ge 1/2$ since $R_x \le 1$.

Let $C_3$ be a constant such that
\[
T :=  \frac{C_3 \delta^{-2} y^{1 - 2 \alpha}}{(2\alpha - 1) \log y} \ge C_2 \delta^{-2} \sum_{p > y} \frac{1}{p^{2 \alpha}}.
\]

We have $R_x \ll (T m)^m$. Note that $T = o(1)$ by assumption. Let $m = [T^{-1} / e]$.
We obtain
\[
R_x \ll \exp \left( - c_0 \delta^2 (2 \alpha - 1) y^{2 \alpha - 1} \log y  \right)
\]
for some $c_0 > 0$.

\end{proof}

\begin{lemma} \label{first bound y}
Suppose that $(\log X)^3 \le y$. Then $\Prob\left(f \in \mathcal{L}_X^+ \mid f(p) = 1 \, (p \le y)\right) = 1 - o(1)$.
\end{lemma}

\begin{proof}
Let us apply Lemma \ref{sum f(n) bound lemma1} with $\delta = 1/10$ at each integer in the interval $[y, X]$. We obtain
\begin{multline*}
1 - \Prob\left(f \in \mathcal{L}_X^+ \mid f(p) = 1 \, (p \le y)\right) \le \sum_{y \le n \le X} R_n \ll \\
X \exp \left(-\frac{c_0}{100} (2\alpha(X, y) - 1) y^{2\alpha(X, y) - 1} \log y \right). 
\end{multline*}
The right-hand side is $o(1)$ if
\begin{equation} \label{X equation alpha}
(2 \alpha(X, y) - 1) y^{2 \alpha(X, y) - 1} \log y \ge K_0 \log X,
\end{equation}
where $K_0$ is sufficiently large. Inequality (\ref{X equation alpha}) follows from the assumption $(\log X)^3 \le y$. This can be shown in the same way as we deduced Theorem \ref{thm sum f(n)} from Lemma \ref{keylemma1}.
\end{proof}

From now on we assume that $x$ is sufficiently large so that $y = (\log x)^{2 + \varepsilon} > y_0, \varepsilon > 0$. In particular, the condition $y \ge y_1$ of Lemma \ref{triangle smooth} is satisfied. Let $\log_2 X = \frac{2}{3} \log_2 x$. Lemma \ref{first bound y} shows that with probability $1 - o(1)$ the partial sums of $f(n)$ are nonnegative up to $X$.

Let $x_0 = X$, $x_{i+1} = x_i + \frac{x_i}{h(x_i)}$, where $(\log x_i)^{2.01} \le h(x_i) \ll (\log x_i)^{100}$ is a monotonically increasing function that will be defined later. There are $O((\log x)^{102})$ points $x_i$. We apply Lemma \ref{sum f(n) bound lemma1} with $\delta = 1/100$ at each $x_i$ and obtain that
\[
\Prob \left( \exists x_i \, \sum_{n \le x_i} f(n) \le 0.99 \Psi^{*}(x_i, y) \right) \ll
(\log x)^{102} \exp\left(-\frac{c_0}{10^4} (2\alpha - 1) y^{2\alpha - 1} \log y \right).
\]
This is $o(1)$ if (\ref{keylemma1equation}) is satisfied with $K$ sufficiently large.

We denote
\[
R := \Prob \left( \exists i \, \exists u \in [x_i, x_{i + 1}] : \, \sum_{x_i < n \le u} f(n) \le -\frac{1}{10} \Psi^{*}(x_i, y) \right).
\]
It is enough to prove that $R = o(1)$ if the assumptions of Lemma \ref{keylemma1} hold.

First let us rewrite $\sum_{x_{i} < n \le u} f(n)$ as
\[
\sum_{x_{i} < n \le u} f(n) = \Psi^*(x_i + u, y) - \Psi^*(x_i, y) + \sum_{\substack{p(n) > y \\ n \ne 1}}{}^{\flat} f(n) \left( \Psi^*\left( \frac{x_i + u}{n}, y \right) - \Psi^*\left( \frac{x_i}{n}, y \right) \right).
\]

Since $\Psi^*(x_i + u, y) - \Psi^*(x_i, y)$ is nonnegative, it is enough to give a good upper bound on
\[
R'_i := \Prob\left(\exists u\in [x_i, x_{i + 1}] : \tilde{\sum}_{x_i < n \le u}f(n) \le -\frac{1}{10} \Psi^*(x_i, y) \right),
\]
where $\tilde{\sum}$ means that the sum is over integers that are not of the form $m^2 s$, where $s$ is $y$-smooth. 

Let
\[
R'_{k, l} := \Prob \left( \left|\tilde{\sum}_{x_i + 2^k l < n \le x_i + 2^{k} (l + 1)} f(n) \right| \ge \frac{(\log (x_{i + 1} - x_i))^{-1}}{50}  \Psi^*(x_i, y) \right).
\]

Let $u \in [x_i, x_{i + 1}]$ be such that $u - x_i$ is a natural number. Let $u - x_i = 2^{\alpha_1} + 2^{\alpha_2} + \ldots + 2^{\alpha_j}$ be the binary expansion, where $\alpha_1 > \alpha_2 > \ldots > \alpha_j$. Of course $j \le \log(u - x_i) / \log 2$. This gives a partition of the interval $(x_i, u]$ into subintervals $(x_i, x_i + 2^{\alpha_1}], (x_i + 2^{\alpha_1}, x_i + 2^{\alpha_1} + 2^{\alpha_2}], \ldots, (x_i + \sum_{\tau \le j - 1} 2^{\alpha_{\tau}}, x_i + \sum_{\tau \le j} 2^{\alpha_{\tau}}]$. 
All of them are of the form $(x_i + 2^k l, x_i + 2^{k} (l + 1)]$. 
Hence
\begin{equation} \label{R'_i and R'_{k, l}}
R'_i \le \sum_{2^k l \le x_{i + 1} - x_i} R'_{k, l}.
\end{equation}

We have
\begin{equation} \label{f(n) 2^k interval sum}
\tilde{\sum}_{x_i + 2^k l < n \le x_i + 2^{k} (l + 1)} f(n) = \sum_{\substack{p(n) > y \\ n \ne 1}}{}^{\flat} f(n) \left( \Psi^*\left( \frac{x_i + 2^k (l + 1)}{n}, y \right) - \Psi^*\left( \frac{x_i + 2^k l}{n}, y \right) \right).
\end{equation}

Lemma \ref{triangle smooth} gives us
\[
\Psi^*\left( \frac{x_i + 2^k (l + 1)}{n}, y \right) - \Psi^*\left( \frac{x_i + 2^k l}{n}, y \right) \le
 \Psi^* \left( \frac{2^k}{n}, y \right) + 
\sqrt{\frac{2 x_{i + 1}}{n}}.
\]

Thus we can fix two sequences $b_{k, l, i}(n)$ and $d_{k, l, i}(n)$ such that
\begin{equation} \label{two sequences div}
\Psi^*\left( \frac{x_i + 2^k (l + 1)}{n}, y \right) - \Psi^*\left( \frac{x_i + 2^k l}{n}, y \right) = b_{k, l, i}(n) + d_{k, l, i}(n), 
\end{equation}
\[
b_{k, l, i}(n) \le \Psi^* \left( \frac{2^k}{n}, y \right), \qquad
 d_{k, l, i}(n) \le \sqrt{\frac{2 x_{i + 1}}{n}}.
\]

Then $R'_{k, l} \le B'_{k, l} + D'_{k, l}$, where
\[
 B'_{k, l} = \Prob \left( \left| \sum_{\substack{p(n) > y \\ n \ne 1}}{}^{\flat} b_{k, l, i}(n) f(n) \right|
 \ge \frac{(\log (x_{i + 1} - x_i))^{-1}}{100}  \Psi^*(x_i, y) \right),
\]
\[
 D'_{k, l} = \Prob \left(  \left| \sum_{\substack{p(n) > y \\ n \ne 1}}{}^{\flat} d_{k, l, i}(n) f(n) \right|
 \ge \frac{(\log (x_{i + 1} - x_i))^{-1}}{100}  \Psi^*(x_i, y) \right).
\]

We apply Lemma \ref{Halasz lemma}, Lemma \ref{alpha ineq} and Markov's inequality to obtain
\[
B'_{k, l} \ll C_4^{m} (\log (x_{i + 1} - x_i))^{2m}  \left(  \left(\frac{2^k}{x_i}\right)^{2\alpha(x_{i + 1}, y)} \sum_{\substack{p(n) > y \\ n \ne 1}}{}^{\flat} n^{-2 \alpha(x_{i + 1}, y)} (2m - 1)^{\omega(n)} \right)^m.
\]
Note that $(2^k/x_{i})^{2\alpha(x_{i+1}, y)} \le 2^k/x_{i} = \frac{2^{k}}{(x_{i + 1} - x_i)} h(x_i)^{-1}$.
Following the proof of Lemma \ref{sum f(n) bound lemma1} we obtain
\begin{equation} \label{B_{k,l}' bound}
B_{k,l}' \ll \exp  \left(-c_1 \frac{(x_{i + 1} - x_i)}{2^k} \frac{h(x_i)}{(\log x_i)^2} (2\alpha - 1) y^{2\alpha - 1} \log y \right).
\end{equation}

Now we provide an upper bound on $D'_{k, l}$.
From Lemma \ref{Konyagin} we deduce that 
\[
\Psi(x_i, y) \ge x_i^{1 - \frac{1}{2 + \varepsilon}} = x_i^{\frac{1}{2} + \frac{\varepsilon}{2(2 + \varepsilon)}}.
\]

This and Lemma \ref{Halasz lemma} with Markov's inequality imply that
\[
D'_{k, l} \ll C_5^m (\log x_i)^{2m} x_i^{-\frac{m \varepsilon}{2 + \varepsilon}} \left(  \sum_{\substack{p(n) > y \\ 1 \ne n \le 2x_i}}{}^{\flat} \frac{(2m - 1)^{\omega(n)}}{n}
\right)^m.
\]

A standard application of Rankin's trick shows that
\[
\sum_{n \le 2x_i} \frac{(2m - 1)^{\omega(n)}}{n} \ll (C_6 \log x_i)^{2m - 1}.
\]

Thus 
\begin{equation} \label{D_{k,l}' bound}
D'_{k, l} \ll C_7^m (\log x_i)^{2m^2 + m} x_i^{-\frac{m \varepsilon}{2 + \varepsilon}}.
\end{equation}

Inequalities (\ref{R'_i and R'_{k, l}}), (\ref{B_{k,l}' bound}) and (\ref{D_{k,l}' bound}) imply that
\begin{multline*}
R'_i \ll  C_7^m h(x_i)^{-1} (\log x_i)^{2m^2 + m} x_i^{1-\frac{m \varepsilon}{2 + \varepsilon}} +
\\
\sum_{2^k \le x_{i + 1} - x_i} \frac{|x_{i + 1} - x_i|}{2^k} \exp  \left(-c_1 \frac{(x_{i + 1} - x_i)}{2^k} \frac{h(x_i)}{(\log x_i)^2} (2\alpha - 1) y^{2\alpha - 1} \log y \right) \ll
\\
 C_7^m h(x_i)^{-1} (\log x_i)^{2m^2 + m} x_i^{1-\frac{m \varepsilon}{2 + \varepsilon}} +  \exp \left( - c_2 (2 \alpha - 1) y^{2 \alpha - 1} \log y  \right).
\end{multline*}

There are no more than $O(h(x) (\log x))$ checkpoints $x_i$. Hence

\[
R \ll h(x) (\log x)  \exp \left( - c_2 (2 \alpha - 1) y^{2 \alpha - 1} \log y  \right) + \sum_{x_i > y}  C_7^m h(x_i)^{-1} (\log x_i)^{2m^2 + m} x_i^{1-\frac{m \varepsilon}{2 + \varepsilon}}.
\]
Let us take $m = 10[\varepsilon^{-1}]$. Note that $(\log x_i)^{2m^2 + m} = o(x_i^{1/10})$ is guaranteed by $\varepsilon \ge 1000 \sqrt{\frac{\log_2 X}{\log X}}$. This is satisfied, because $\varepsilon > 1000 \frac{(\log_2 x)^{1/2}}{(\log x)^{1/3}} \ge 1000 \sqrt{\frac{\log_2 X}{\log X}}$.
Hence
\[
R \ll h(x) (\log x)  \exp \left( - c_2 (2 \alpha - 1) y^{2 \alpha - 1} \log y  \right) + o(y^{-1}).
\]
Now we take $h(x) = (\log x)^3$. We see that $R = o(1)$ if $K$ in (\ref{keylemma1equation}) is large enough.
This finishes the proof of Lemma \ref{keylemma1}. \qed

\section{Proof of Theorem \ref{Kerr and Klurman improvement} and Theorem \ref{cor from conj}}

\subsection{Lower bound for $\sum_{n \le x} g(n)$}

\begin{prop} \label{g(n) final lower bound}
    Let $f$ be a random completely multiplicative function such that for each prime $p$ we have $\Prob(f(p) = 1) = \Prob(f(p) = -1) = 1/2$. Let $g = f * 1$.

    Then there exist $c > 0$ and $\beta > 0$ such that 
    \begin{equation} \label{g(n) upper bound with C}
        \Prob \left(
        \sum_{n \le x} g(n) \ge c x \exp \left( \sum_{p \le x} \frac{f(p)}{p} \right)
        \right) = 1 - O \left( \exp \left( - x^{\beta} \right) \right).
    \end{equation}
\end{prop}
One can see from the proof that any fixed $\beta < e^{-2}$ is admissible in Proposition \ref{g(n) final lower bound}.
Our proof of Proposition \ref{g(n) final lower bound} is based on \cite[Proposition 3.3]{kerr2022negative} by Kerr and Klurman, which in turn is based on a theorem by Matom\"{a}ki and Shao \cite[Hypothesis P]{matomaki2020sieve}.

\begin{lemma} \label{Kerr and Klurman Proposition 3.3}
    Let $\varepsilon > 0$ be sufficiently small. Let $f$ be a multiplicative function with $-1 \le f(n) \le 1$ for all $n$. Let $g = 1*f$. For $0 < \delta < 1$, let $\mathcal{P}_{\delta} = \{ p \, \text{prime} \, : \, f(p) \ge -\delta \}$, and suppose for some 
    \[
    \frac{40000}{\varepsilon^2} \le v \le \frac{\log x}{1000 \log_2 x}
    \]
    we have
    \begin{equation*} \label{sum 1/p large}
    \sum_{\substack{p \in \mathcal{P}_{\delta} \\ x^{1/v} \le p \le x}} \frac{1}{p} \ge 1 + \varepsilon.
    \end{equation*}
    Then
    \begin{equation} \label{Kerr and Klurman Proposition 3.3 conclusion}
    \sum_{n \le x} g(n) \gg \varepsilon^4 \left(\frac{(1 - \delta)}{v} \right)^{v(1 + o(1)) / e} \exp \left( \sum_{p \le x} \frac{f(p)}{p} \right) x.
    \end{equation}
\end{lemma}

\begin{proof}
    See \cite[Proposition 3.3]{kerr2022negative}.
\end{proof}

\begin{lemma}[Hoeffding's inequality] \label{Hoeffding's inequality}
    Let $X_1, X_2, \ldots, X_n$ be independent random variables such that $a_i \le X_i \le b_i$ almost surely. Let
    \[
    S_n = X_1 + \ldots + X_n.
    \]
    Then
    \[
    \Prob \left( |S_n - \E[S_n]| \ge t \right) \le
    2 \exp \left( - \frac{2 t^2}{\sum_{i = 1}^n (b_i - a_i)^2} \right).
    \]
\end{lemma}
\begin{proof}
    See \cite[Theorem 2]{hoeffding1963probability}.
\end{proof}

\begin{proof}[Proof of Proposition \ref{g(n) final lower bound}]
In Lemma \ref{Kerr and Klurman Proposition 3.3} we fix $\varepsilon, \delta$ and $v = \max(e^{2 + 4\varepsilon}, 40000 \varepsilon^{-2})$. Let $v_0 = e^{2 + 4 \varepsilon}$. Then for sufficiently large $x$ 
\begin{multline*}
\Prob \left( \sum_{\substack{p \in \mathcal{P}_{\delta} \\ x^{1/v} \le p \le x}} \frac{1}{p} \le 1 + \varepsilon \right) \le
\Prob \left( \sum_{\substack{p \in \mathcal{P}_{\delta} \\ x^{1/v_0} \le p \le x}} \frac{1}{p} \le 1 + \varepsilon \right)  
\le \\
\Prob \left( \left | \sum_{x^{1/v_0} \le p \le x} \frac{f(p)}{p} \right| \ge \varepsilon \right) \ll \exp(-x^{1/v_0}).
\end{multline*}
Here we used Hoeffding's inequality (Lemma \ref{Hoeffding's inequality}). Hence the conditions of Lemma \ref{Kerr and Klurman Proposition 3.3} are satisfied with probability $1 - O(\exp(-x^{1/v_0}))$.
The equation (\ref{Kerr and Klurman Proposition 3.3 conclusion}) implies 
\[
\sum_{n \le x} g(n) \gg x \exp \left( \sum_{p \le x} \frac{f(p)}{p} \right),
\]
since $\varepsilon, \delta$ and $v$ are fixed.

\end{proof}

\begin{rem}
Let us sketch an alternative proof. 
A result of Tenenbaum \cite[Theorem 1.2]{tenenbaum2017moyennes} implies the following.
Let $g$ be a nonnegative multiplicative function such that $g(p^k) \le k$ for all primes $p$, and let $\varrho > 0, \varepsilon > 0$, and $x > x_0$. Suppose that, uniformly in $y$,
\begin{equation} \label{sum g(n) lower bound condition}
\sum_{p \le y} \frac{(g(p) - \varrho) \log p}{p} \ll \varepsilon \log y 
\quad
(x^{\varepsilon} < y \le x).
\end{equation}
Then
\begin{equation} \label{sum g(n) lower bound}
\sum_{n \le x} g(n) \gg \frac{x}{\log x} \exp \left( \sum_{p \le x} \frac{g(p)}{p}\right).
\end{equation}

It remains to show that (\ref{sum g(n) lower bound condition}) holds with admissible probability. This can be done using Lemma \ref{Hoeffding's inequality}.
\end{rem}

\subsection{Applying Rankin's trick} \label{steps like in Kerr and Klurman}

In this section we follow the main steps of the proof of \cite[Proposition 5.6]{kerr2022negative}.

Proposition \ref{g(n) final lower bound} implies that if $c_1$ is sufficiently small, then we have
\[
P_x \le \Prob \left( \left| \sum_{n \le x} f(n) \left\{ \frac{x}{n} \right\} \right| > c_1 x  \exp \left( \sum_{p \le x} \frac{f(p)}{p} \right) \right) + O(\exp(-x^{\beta})).
\]

Let
\begin{equation} \label{S moment def}
S := \E \left[  \left( x^{-1} \exp \left( -\sum_{p \le x} \frac{f(p)}{p} \right) \sum_{n \le x} f(n) \left\{ \frac{x}{n} \right\} \right)^q \right]^{1/q},
\end{equation}
where $q > 0$ is an even integer.

Suppose that for $q = q(x)$ we have $S = o(1)$.
Markov's inequality implies
\begin{equation} \label{P_x moment ineq}
P_x \le (c_1^{-1} S)^q + O(\exp(-x^{\beta})) \ll \exp(-q) + O(\exp(-x^{\beta})).
\end{equation}

That is why we want to prove that $S = o(1)$ for $q$ as large as possible.

Let $\frac{1}{\log_2 x} \ll \varepsilon = \varepsilon(x) = o(1)$.
We partition the summation over $n$ as
\[
\sum_{n \le x} f(n) \left\{ \frac{x}{n} \right\} = \sum_{\substack{nl \le x \\ n \in R(x, x^{\varepsilon}) \\ l \in S(x, x^{\varepsilon})}} f(l) f(n) \left\{ \frac{x}{n l} \right\} = 
\sum_{j \le \log x + 1} \sum_{\substack{nl \le x \\ n \in R(x, x^{\varepsilon}) \\ l \in S(x, x^{\varepsilon}) \\ e^j \le l < e^{j + 1}}} f(l) f(n) \left\{ \frac{x}{n l} \right\}.
\]

Let $h_1(x) = o(\varepsilon \log x)$ be a function to be chosen later.

Applying Minkowski's inequality, we obtain $S \le S_1 + S_2$, where
\[
S_1 = x^{-1} \sum_{j \le \log x - h_1(x)} \E \left[ 
\left( 
\exp \left( -\sum_{p \le x} \frac{f(p)}{p} \right)
\sum_{\substack{nl \le x \\ n \in R(x, x^{\varepsilon}) \\ l \in S(x, x^{\varepsilon}) \\ e^j \le l < e^{j + 1}}} f(l) f(n) \left\{ \frac{x}{n l} \right\}
\right)^q
\right]^{1/q},
\]
\[
S_2 = x^{-1} \sum_{\log x - h_1(x) < j \le \log x + 1}
\E \left[
\left(
\exp \left( -\sum_{p \le x} \frac{f(p)}{p} \right)
\sum_{\substack{nl \le x \\ n \in R(x, x^{\varepsilon}) \\ l \in S(x, x^{\varepsilon}) \\ e^j \le l < e^{j + 1}}} f(l) f(n) \left\{ \frac{x}{n l} \right\}
\right)^q
\right]^{1/q}.
\]

 We first evaluate $S_1$, following the method used to bound the analogous sum in the proof of \cite[Proposition 5.6]{kerr2022negative}. We use the majorant principle with Rankin's trick. By the majorant principle, we mean that when all coefficients $a_n \ge 0$, the moment
\[
\E \left[ \left( \sum_n a_n f(n) \right)^q \right] = \sum_{\substack{n_1, n_2, \ldots, n_q \\ n_1 n_2 \ldots n_q = \square}} a_{n_1} a_{n_2} \ldots a_{n_q}
\]
is monotone non-decreasing in each individual coefficient $a_n$.
Since in the expansion $\exp \left( -\sum_{p \le x} \frac{f(p)}{p} \right) = \sum_n a_n f(n)$ some of the coefficients $a_n$ are negative, which prevents us from using the majorant principle immediately. That is why we use a trivial upper bound
$\exp \left( -\sum_{p \le x} \frac{f(p)}{p} \right) \ll  \log x$. This provides

\[
S_1 \ll \frac{\log x}{x} \sum_{j \le \log x - h_1(x)} \E \left[ 
\left( 
\sum_{\substack{nl \le x \\ n \in R(x, x^{\varepsilon}) \\ l \in S(x, x^{\varepsilon}) \\ e^j \le l < e^{j + 1}}} f(l) f(n) \left\{ \frac{x}{n l} \right\}
\right)^q
\right]^{1/q}.
\]
Now we apply the majorant principle to obtain
\begin{equation} \label{S_1 eq 1}
S_1 \ll \frac{\log x}{x} \sum_{j \le \log x - h_1(x)} \E \left[ \left( 
\sum_{\substack{n \le x/e^j \\ n \in R(x, x^{\varepsilon})}} 
f(n)
\sum_{\substack{l \le e^{j + 1} \\ l \in S(x, x^{\varepsilon})}} 
f(l)
\right)^q \right]^{1/q}.
\end{equation}

Note that if $n \le x/e^j$ and $l \le e^{j + 1}$, then for $0 < \delta < 1$
\[
\left( \frac{x}{nl} \right) \left( \frac{e^j}{x} \right)^{\delta} n^{\delta} \ge e^{-1} \left(\frac{x}{n e^j}\right) \left( \frac{n e^j}{x} \right)^{\delta} \ge e^{-1}.
\]
This inequality, together with the majorant principle, gives us
\[
S_1 \ll (\log x) x^{- \delta} \sum_{j \le \log x - h_1(x)} e^{j \delta} 
\E \left[
\left( \sum_{\substack{n \le x/e^j \\ n \in R(x, x^{\varepsilon})}} \frac{f(n)}{n^{1 - \delta}}
\sum_{\substack{l \le e^{j + 1} \\ l \in S(x, x^{\varepsilon})}} \frac{f(l)}{l}
\right)^q
\right]^{1/q}.
\]

Hence
\begin{multline} \label{S_1 prefinal bound}
S_1 \ll (\log x) x^{- \delta}
\delta^{-1} e^{\delta(\log x - h_1(x))}
\E \left[ \left(
\sum_{n \in R(x, x^{\varepsilon})} \frac{f(n)}{n^{1 - \delta}} \sum_{l \in S(x, x^{\varepsilon})} \frac{f(l)}{l} \right)^q
\right]^{1/q} \ll
\\
(\log x) \delta^{-1} e^{-\delta h_1(x)} 
\E \left[ \left( \prod_{p > x^{\varepsilon}} \left( 1 - \frac{f(p)}{p^{1 - \delta}} \right)^{-1 } \right)^q \right]^{1/q} 
\E \left[ \left( 
\prod_{p \le x^{\varepsilon}}
\left( 1 - \frac{f(p)}{p} \right)^{-1}
\right)^q \right]^{1/q}.
\end{multline}

Let $|z| < 1/2$. For all such $z$ we have 
\[
\log(1 + z) \ge -z^2 + z.
\]
Thus
\begin{equation} \label{elementary inequality log}
 (1 + z)^{-q} \le \exp\left( q z^2 \right) \exp(-qz).   
\end{equation}

Suppose that $\delta \le 1/3$. Applying inequality (\ref{elementary inequality log}) twice for $z = p^{\delta - 1}$ and $z = -p^{\delta - 1}$, we obtain for $p \ge 3$

\begin{multline*}
\E \left[ \left( 1 - \frac{f(p)}{p^{1 - \delta}} \right)^{-q} \right] = 
\frac{1}{2} \left(
\left( 1 + \frac{1}{p^{1 - \delta}} \right)^{-q} + \left(
1 - \frac{1}{p^{1 - \delta}}
\right)^{-q}
\right) \le 
\\
\exp(q p^{-4/3}) \frac{\exp(q / p^{1 - \delta}) + \exp(-q/p^{1 - \delta})}{2} \le 
\exp(q p^{-4/3}) \exp \left( \frac{q^2}{2 p^{2 - 2\delta}} \right).
\end{multline*}

The last inequality follows from
\[
\frac{e^z + e^{-z}}{2} \le e^{\frac{z^2}{2}},
\]
which holds for all $z \in \mathbb{R}$.

Thus
\begin{equation} \label{S_1 rough E}
\E \left[ \left( \prod_{p > x^{\varepsilon}} \left( 1 - \frac{f(p)}{p^{1 - \delta}} \right)^{-1 } \right)^q \right] \ll
\exp \left( \frac{c_2 q^2}{\varepsilon (\log x) x^{(1 - 2 \delta) \varepsilon}} 
+ O(q)
\right),
\end{equation}

Also
\begin{multline} \label{S_1 smooth E}
\E \left[ \left( 
\prod_{p \le x^{\varepsilon}}
\left( 1 - \frac{f(p)}{p} \right)^{-1}
\right)^q \right]
\ll \exp\left(q \sum_{p \le x^{\varepsilon}} \frac{1}{p} + O(q) \right) 
\ll \\
\exp \left( q \log \log x + q \log \varepsilon + O(q) \right).
\end{multline}

Combining (\ref{S_1 prefinal bound}), (\ref{S_1 rough E}) and (\ref{S_1 smooth E}), we deduce
\begin{multline} \label{S_1 final bound}
S_1 \ll (\log x) \delta^{-1} e^{-\delta h_1(x)} 
\exp
\left( 
\frac{c_2 q}{\varepsilon (\log x) x^{(1 - 2 \delta) \varepsilon}} + \log \log x + \log \varepsilon
\right) \ll 
\\
 \frac{\varepsilon}{\delta}  \left(\log x \right)^2 
e^{-\delta h_1(x)} \exp \left( \frac{c_2 q}{\varepsilon (\log x) x^{(1 - 2 \delta)\varepsilon}} \right).
\end{multline}

Now let us estimate $S_2$. There exists $j_0$ satisfying $\log x - h_1(x) < j_0 \le \log x + 1$ such that
\begin{equation} \label{S_2 eq 1}
S_2 \ll x^{-1} h_1(x) 
\E \left[
\left(
 \exp \left( -\sum_{p \le x} \frac{f(p)}{p} \right)
\sum_{\substack{nl \le x \\ n \in R(x, x^{\varepsilon}) \\ l \in S(x, x^{\varepsilon}) \\ e^{j_0} \le l < e^{j_0 + 1}}} f(l) f(n) \left\{ \frac{x}{n l} \right\}
\right)^q
\right]^{1/q}.
\end{equation}

In (\ref{S_2 eq 1}) we have $n \ll e^{h_1(x)} = o(x^{\varepsilon})$. But $n \in R(x, x^{\varepsilon})$ and thus $n = 1$.

We conclude that 
\[
S_2 \ll x^{-1} h_1(x) 
\E \left[
\left(
\exp \left( -\sum_{p \le x} \frac{f(p)}{p} \right)
\sum_{\substack{l \in S(x, x^{\varepsilon}) \\ e^{j_0} \le l < e^{j_0 + 1}}} f(l)  \left\{ \frac{x}{l} \right\}
\right)^q
\right]^{1/q},
\]
and hence
\begin{equation} \label{S_2 prefinal bound}
S_2 \ll x^{-1} h_1(x) \, \sup_{f} \left| 
\exp \left( -\sum_{p \le x} \frac{f(p)}{p} \right)
\sum_{\substack{l \in S(x, x^{\varepsilon}) \\ e^{j_0} \le l < e^{j_0 + 1}}} f(l)  \left\{ \frac{x}{l} \right\} \right|,
\end{equation}
where the supremum is over the set of completely multiplicative functions that take values in $\{ 1, -1 \}$.

Let us denote
\[
\Psi_f(x, y) := \sum_{n \in S(x, y)} f(n).
\]

We will provide an upper bound on $\Psi_f(x, y)$.

\subsection{Upper bound on $\Psi_f(x, y)$}

\begin{prop} \label{Psi_f proposition}
Let $f(n)$ be a completely multiplicative function such that $|f(n)| \le 1$ for all $n$.
Let $a \ge 0$ and $u_x := \frac{\log x}{\log y}$.

Let $y \ge 2$ and suppose that uniformly for $y^{a} \le t \le y^{a + 1}$ we have
\begin{equation}\label{Psi_f upper bound assumption}
|\Psi_f(t, y)| \le c_f(y) \rho \left( u_t \right) t.
\end{equation}
Also suppose that $c_f(y) \gg y^{-1/7}$.

Then for any $\varepsilon > 0$ uniformly in the range $x \ge y^a$, $\log y \ge (\log_2 x)^{5/3 + \varepsilon}$, we have
\begin{equation} \label{Psi_f upper bound}
|\Psi_f(x, y)| \le c_f(y) x \rho(u_x) \left( 1 + O_{\varepsilon} \left( \frac{u_x \log(u_x + 1)}{\log x} \right) \right).
\end{equation}
\end{prop}

In the following Lemma we collect the properties of the Dickman function that we will need.

\begin{lemma}\label{Dickman function}
(i) $\rho(u) u = \int_{u - 1}^u \rho(t) \, dt \,\, (u \ge 1)$,

(ii) Uniformly for $y \ge 1.5$ and $1 \le u \le \sqrt{y}$ we have
\[
\int_0^{u} \rho(u - t) y^{-t} \, dt \ll \frac{\rho(u)}{\log y},
\]

(iii) 
Uniformly for $y \ge 1.5$ and $1 \le u \le \sqrt{y}$ we have
\[
\int_1^{u} \rho(u - t) y^{-t} \, dt \ll \frac{\rho(u)}{(\log y) y^{1/3}},
\]

(iv) 
Uniformly for $y \ge 1.5$ and $1 \le u \le y^{1/4}$ we have
\[
\sum_{\substack{y < p^m \le y^u \\ p \le y}}
\frac{\log p}{p^m} \rho \left(
u - \frac{\log p^m}{\log y}\right)
\ll \rho(u) \frac{\log y}{y^{1/6}}.
\]

(v) 
For every fixed $\varepsilon > 0$ and uniformly for $y \ge 1.5, u \ge 1$ and $0 \le \theta \le 1$ we have
\begin{multline*}
\sum_{p^m \le y^{\theta}} \frac{\log p}{p^m} \rho \left( u - \frac{\log p^m}{\log y} \right) = (\log y) \int_{u - \theta}^{u} \rho(t) \, dt \, + \\
O_{\varepsilon} (\rho(u) \left\{ 1 + u \log^2(u + 1) \exp (- (\log y)^{3/5 - \varepsilon}) \right\}).
\end{multline*}

\end{lemma}

\begin{proof}
For (i) see, for example, \cite[Lemma 1 (ii)]{hildebrand1986number}.
(ii) is \cite[Lemma 2]{hildebrand1986number}, and (iii) easily follows from the proof of \cite[Lemma 2]{hildebrand1986number}.
(iv) follows from the proof of \cite[Lemma 3]{hildebrand1986number}, and (v) is \cite[Lemma 4]{hildebrand1986number}.
\end{proof}

\begin{proof}[Proof of Proposition \ref{Psi_f proposition}]
The formula 
\begin{equation} \label{Psi approximation}
\Psi(x, y) = x \rho(u) \left( 1 + O_{\varepsilon} \left( \frac{ u \log(u + 1)}{\log x} \right) \right),
\end{equation}
where $\Psi(x, y) := |S(x, y)|$
was proved by Hildebrand \cite[Theorem 1]{hildebrand1986number} in the range $\log y \ge (\log_2 x)^{5/3 + \varepsilon}$.

The proof uses the identity
\begin{equation} \label{Psi Buchstab alternative}
\Psi(x, y) \log x = \int_1^x \frac{\Psi(t, y)}{t} \, dt + \sum_{\substack{p^m \le x \\ p \le y}} \Psi \left( \frac{x}{p^m}, y \right) \log p.
\end{equation}

The estimate is derived by an inductive argument provided by (\ref{Psi Buchstab alternative}).

Let $S = \sum_{n \in S(x, y)} f(n) \log n$. Integrating by parts, we obtain
\[
S = \Psi_f(x, y) \log x - \int_{1}^{x} \frac{\Psi_f(t, y)}{t} \, dt.
\]
On the other hand
\[
S = \sum_{n \in S(x, y)} f(n) \sum_{p^m | n} \log p = 
\sum_{\substack{p^m \le x \\ p \le y}} f(p^m) \Psi_f \left( \frac{x}{p^m}, y \right) \log p.
\]
Here we used that by assumption $f$ is completely multiplicative.

Hence the analog of (\ref{Psi Buchstab alternative}) is
\[
\Psi_f(x, y) \log x = \int_{1}^x \frac{\Psi_f(t, y)}{t} dt + \sum_{\substack{p^m \le x \\ p \le y}} f(p^m) \Psi_f \left( \frac{x}{p^m}, y \right) \log p,
\]
which implies
\begin{equation} \label{Psi_f Buchstab alternative}
|\Psi_f(x, y)| \log x \le \int_1^x \frac{|\Psi_f(t, y)|}{t} \, dt + \sum_{\substack{p^m \le x \\ p \le y}} \left|\Psi_f \left( \frac{x}{p^m}, y \right) \right| \log p.
\end{equation}

For $u \ge a$ let $\Delta(y, u)$ be the minimal nonnegative real number such that the inequality
\[
|\Psi_f(y^u, y)| \le c_f(y) y^{u} \rho(u) (1 + \Delta(y, u))
\]
holds. Also denote $\Delta^*(y, u) := \sup_{\max(a, u - 1) \le u' \le u} \Delta(y, u')$, which is well defined for $u \ge a$. Finally let us denote $\Delta^{**}(y, u) := \sup_{a \le u' \le u} \Delta(y, u')$.
We will prove by induction that $\Delta^{**}(y, u) \ll_{\varepsilon} \log(u + 1) / \log y$.

By the assumption (\ref{Psi_f upper bound assumption}) we have $\Delta(y, u) = 0 \,\, (a \le u \le a + 1)$.

The inequality (\ref{Psi_f Buchstab alternative}) and the trivial upper bound $|\Psi_f(t, y)| \le \Psi(t, y)$ imply that for $a + 1 \le u \le \exp((\log y)^{3/5 - \varepsilon})$ we have
\begin{multline*}
\frac{|\Psi_f(y^u, y)|}{\rho(u) y^{u}} \le 
\frac{1}{\rho(u) y^{u} \log y^u} \int_{y^{u - 1}}^{y^u} \frac{|\Psi_f(t, y)|}{t} \, dt \, + \, 
\frac{1}{\rho(u) y^{u} \log y^u} \int_{1}^{y^{u-1}} \frac{\Psi(t, y)}{t} \, dt 
\, + \\
\frac{1}{\rho(u) y^u \log y^u} \sum_{\substack{\sqrt{y} < p^m \le y}} \left| \Psi_f\left( \frac{y^u}{p^m}, y \right) \right| \log p
\, + \\
\frac{1}{\rho(u) y^u \log y^u} \sum_{\substack{p^m \le \sqrt{y}}} \left| \Psi_f\left( \frac{y^u}{p^m}, y \right) \right| \log p
+ 
\frac{1}{\rho(u) y^u \log y^u} 
\sum_{\substack{y < p^m \le y^u \\ p \le y}} \Psi\left( \frac{y^u}{p^m}, y \right) \log p,
\end{multline*}

Now we use the definition of $\Delta(y, u)$ and formula (\ref{Psi approximation}) to obtain

\begin{multline} \label{Psi_f eq 0}
\frac{|\Psi_f(y^u, y)|}{\rho(u) y^{u}} \le 
\frac{c_f(y) (1 + \Delta^*(y, u))}{\rho(u) y^{u} \log y^u} \int_{y^{u - 1}}^{y^u} \rho \left( \frac{\log t}{\log y} \right) \, dt \, + \,
\frac{O(1)}{\rho(u) y^{u} \log y^u} \int_{1}^{y^{u-1}} \rho \left( \frac{\log t}{\log y} \right) \, dt \,
+ \\
\frac{c_f(y) (1 + \Delta^*(y, u - 1/2))}{\rho(u) \log y^u} \sum_{\substack{\sqrt{y} < p^m \le y}} 
\frac{\log p}{p^m} \rho \left( u - \frac{\log p^m}{\log y} \right)
\, + \\
\frac{c_f(y) (1 + \Delta^*(y, u))}{\rho(u) \log y^u} \sum_{\substack{p^m \le \sqrt{y}}} 
\frac{\log p}{p^m} \rho \left( u - \frac{\log p^m}{\log y} \right)
\, + \,
\frac{O(1)}{\rho(u) \log y^u} 
\sum_{\substack{y < p^m \le y^u \\ p \le y}} 
\frac{\log p}{p^m} \rho \left( u - \frac{\log p^m}{\log y} \right).
\end{multline}

By Lemma \ref{Dickman function} (ii) we have
\begin{equation} \label{Psi_f eq 1}
\int_{y^{u - 1}}^{y^u} \rho \left( \frac{\log t}{\log y} \right) dt = (\log y) y^u \int_{0}^1 \rho(u - \tau) y^{-\tau} \, d\tau \ll y^u \, \rho(u).
\end{equation}

Lemma \ref{Dickman function} (iii) implies
\begin{equation} \label{Psi_f eq 2}
\int_{1}^{y^{u-1}} \rho \left( \frac{\log t}{\log y} \right) \, dt = (\log y) y^u \int_{1}^u \rho(u - \tau) y^{-\tau} \, d\tau \ll \frac{y^u \rho(u)}{y^{1/3}}.
\end{equation}

By part (i) of Lemma \ref{Dickman function},
\begin{equation} \label{Psi_f eq alpha}
1 = \frac{1}{\rho(u) u} \int_{u - 1/2}^u \rho(t) \, dt + \frac{1}{\rho(u) u} \int_{u-1}^{u - 1/2} \rho(t) \, dt =: \alpha(u) + (1 - \alpha(u)).
\end{equation}

Part (v) of Lemma \ref{Dickman function}, equation (\ref{Psi_f eq alpha}) and the assumption $u \le \exp((\log y)^{3/5 - \varepsilon})$ imply
\begin{multline} \label{Psi_f eq 3}
\sum_{\substack{\sqrt{y} < p^m \le y}} 
\frac{\log p}{p^m} \rho \left( u - \frac{\log p^m}{\log y} \right) = (\log y) \int_{u - 1}^{u - \frac{1}{2}} \rho(t) \, dt + \\
O_{\varepsilon}\left(\rho(u) \left\{ 1 + u \log^2(u + 1) \exp(-(\log y)^{3/5 - \varepsilon/2}) \right\} \right) = \\
\rho(u) \log y^u \left( (1 - \alpha(u)) + O_{\varepsilon}\left( \frac{1}{\log y^u} \right) \right).
\end{multline}

In exactly the same way we get
\begin{equation} \label{Psi_f eq 4}
\sum_{\substack{p^m \le \sqrt{y}}} 
\frac{\log p}{p^m} \rho \left( u - \frac{\log p^m}{\log y} \right) = \rho(u) \log y^u \left( \alpha(u) + O_{\varepsilon}\left( \frac{1}{\log y^u} \right) \right).
\end{equation}

Applying (\ref{Psi_f eq 1}), (\ref{Psi_f eq 2}), (\ref{Psi_f eq 3}), (\ref{Psi_f eq 4}) and Lemma \ref{Dickman function} (iv) to the corresponding terms on the right-hand side of (\ref{Psi_f eq 0}), we derive an estimate

\begin{multline*}
\frac{|\Psi_f(y^u, y)|}{\rho(u) y^{u}} \le 
\frac{c_f(y) O(1) (1 + \Delta^*(y, u))}{u \log y} \, + \,
O \left( \frac{1}{u y^{1/3}} \right)
\, + \\
c_f(y) (1 + \Delta^*(y, u - 1/2)) \left( (1 - \alpha(u)) \, + \, O_{\varepsilon} \left( \frac{1}{u \log y} \right)\right) \, + 
\\
c_f(y) (1 + \Delta^*(y, u)) \left( \alpha(u) + O_{\varepsilon} \left( \frac{1}{u \log y} \right)\right) \,
+ \,
O \left( \frac{1}{u y^{1/6}} \right).
\end{multline*}

Since $\rho(u)$ is a nonincreasing function of $u$, we have $\alpha(u) \le (1 - \alpha(u))$ and hence $\alpha(u) \le 1/2$.

We obtain
\begin{equation} \label{Psi 1/2 eq}
\frac{|\Psi_f(y^u, y)|}{\rho(u) y^{u}} \le
c_f(y) \left( 1 + \frac{1}{2}\Delta^{**}(y, u) + \frac{1}{2}\Delta^{**}(y, u - 1/2)  + 
O_{\varepsilon} \left( \frac{(1 + \Delta^{**}(y, u))}{u \log y} \right) \right).
\end{equation}

By changing the constant in $O_{\varepsilon}$, we can assume that the right-hand side of (\ref{Psi 1/2 eq}) is an upper bound for $\Delta^{*}(y, u)$.

We find that
\[
\Delta^{**}(y, u) \left(1 + O_{\varepsilon} \left( \frac{1}{u \log y} \right) \right) \le \Delta^{**}(y, u - 1/2) + O_{\varepsilon}\left( \frac{1}{u \log y} \right).
\]
By induction
\[
\Delta^{**}(y, u) \ll_{\varepsilon} \frac{\log(u + 1)}{\log y}.
\]

This finishes the proof of Proposition \ref{Psi_f proposition}.
\end{proof}

\subsection{Corollary from Halász's theorem}

\begin{lemma} \label{sum_{p < x} to sum_p} \,

(i)
Let $|a_p| \le 1$ for all $p$ and $\alpha = 1 + \frac{1}{\log x}$. Then
\[
\sum_{p \le x} \frac{a_p}{p} = \sum_p \frac{a_p}{p^{\alpha}} + O(1).
\]

(ii) 
Let $\operatorname{Re}(s) > 1$ and let $F(s) := \sum_n \frac{f(n)}{n^s}$, where $f$ is a multiplicative function taking values in the unit disk and such that $f(2) = 0$. Then
\[
\log F(s) = \sum_{p} \frac{f(p)}{p^s} + O(1).
\]

(iii) Let $s = \sigma + i t$. In the region $\sigma \ge 1, |t| \ge 2$
\[
\frac{1}{\zeta(s)} \ll (\log |t|)^7.
\]
\end{lemma}

\begin{proof}
    (i) follows from  Chebyshev's upper bound on the prime counting function. For (iii) see \cite[Section 3.6]{titchmarsh1986theory}. See \cite[Section 6.19]{titchmarsh1986theory} for a better upper bound. To prove (ii) note that for $\operatorname{Re}(s) > 1$
    \begin{multline*}
    \log F(s) = \sum_p \log \left(1 + \sum_{k \ge 1} \frac{f(p^{k})}{p^{ks}} \right) = \sum_{p} \frac{f(p)}{p^s} + \\
    \sum_p \sum_{k \ge 2} \frac{f(p^k)}{p^{ks}} + \sum_{p} O \left( \left( \sum_{k \ge 1} \frac{f(p^{k})}{p^{ks}} \right)^2 \right) 
    = \sum_{p} \frac{f(p)}{p^s} + O(1).
    \end{multline*}
    Here we used that $\left|\sum_{k \ge 1} \frac{f(p^k)}{p^{ks}} \right| \le 1/2$ for all $p$ and $\log(1 + x) = x + O(x^2)$ in the range $|x| \le 1/2$.
\end{proof} 

The following lemma is a form of Halász's theorem by Granville, Harper and Soundararajan  \cite{granville2018more}.

\begin{lemma}[Halász's theorem] \label{Halász}
Let $f$ be a multiplicative function such that $|f(n)| \le 1$ for all $n$. Let 
\[
F_x(s) := \prod_{p \le x} \left(
1 + \sum_{k = 1}^{\infty} \frac{f(p^k)}{p^{ks}}
\right).
\]
Let 
\[
L(x) := \left( \sum_{|N| \le (\log x)^2 + 1} \frac{1}{N^2 + 1} \sup_{|t - N| \le 1/2} |F_x(1 + it)|^2 \right)^{1/2}.
\]

Then

(i) 
\[
\sum_{n \le x} f(n) \ll x \frac{L(x)}{\log x} \log \left( 100 \frac{\log x}{L(x)} \right) + x \frac{\log_2 x}{\log x}.
\]

(ii) If the multiplicative function $f(n)$ is supported only on $Cx^{0.99}$-smooth numbers, then 
\[
\sum_{n \le x} f(n) \ll \frac{x}{\log x} (L(x) + 1).
\]
\end{lemma}

\begin{proof}
Part (i) is \cite[Theorem 1]{granville2018more}. The proof of (ii) is sketched in \cite[Remark 3.2]{granville2018more}. 
Let us discuss the details.

Following the proof of \cite[Theorem 1]{granville2018more}, one deduces
\begin{multline*}
\sum_{n \le x} f(n) = \frac{1}{\log x}\sum_{k = 1}^{10} S_k(x) + O \left( \frac{x}{\log x} \sum_{\log^4 x < p \le C x^{0.99}} \frac{\log p}{p \log(x/p)} \right)+ \\
\frac{1}{\log x} \sum_{p \le \log^4 x} f(p) \log p \sum_{m \le x/p} f(m) + O\left( \frac{x}{\log x} \right),
\end{multline*}
where
\[
S_k(x) = \sum_{\substack{pqn \le x \\ x^{1 - e^{1-k}} < p \le x^{1 - e^{-k}}}} \frac{f(p) \log p}{\log(x/p)} f(n)f(q) \log q.
\]
It is proved in \cite{granville2018more} that $S_k(x) \ll x L(x) + x$. Also note that 
\[
\frac{x}{\log x} \sum_{\log^4 x < p \le C x^{0.99}} \frac{\log p}{p \log(x/p)} \ll \frac{x}{\log x}.
\]
Thus it is enough to prove that
\[
 \sum_{p \le \log^4 x} f(p) \log p \sum_{m \le x/p} f(m) \ll x L(x) + x.
\]
For $p \le \log^4 x$ we have $F_x(s) \asymp F_{x/p}(s)$ if $\operatorname{Re}(s) \ge 1$. Hence $L(x/p) \ll L(x)$ and Lemma \ref{Halász}(i) implies that
\[
 \sum_{p \le \log^4 x} f(p) \log p \sum_{m \le x/p} f(m) \ll 
 \sum_{p \le \log^4 x} \log p  \frac{x (L(x)+1)}{p \log x} \log_2(x) \ll x (L(x) + 1) \frac{(\log_2 x)^2}{\log x}.
\]
This finishes the proof of part (ii).
\end{proof}

\begin{prop} \label{sum f(n) upper bound in terms of sum f(p)/p}
Let $x \ge 1$ and let $f$ be a real-valued multiplicative function supported on $C x^{0.99}$-smooth numbers and such that $|f(n)| \le 1$ for all $n$. Then
\[
\sum_{n \le x} f(n) \ll x \exp \left( \sum_{p \le x} \frac{f(p)}{p} \right).
\]
\end{prop}

\begin{proof}
By $\nu_2(n)$ we denote the largest integer $v$ such that $2^v | n$. Since
\[
\sum_{n \le x} f(n) = \sum_{\substack{n \le x \\ \nu_2(n) \ne 1}} f(n) + f(2) \sum_{\substack{n \le x/2 \\ (n, 2) = 1}} f(n),
\]
it suffices to prove Proposition \ref{sum f(n) upper bound in terms of sum f(p)/p} for $f$ with $f(2) = 0$. 

Lemma \ref{sum_{p < x} to sum_p} (ii) implies that $|F_x(s)| \asymp \exp \left(\sum_{p \le x} \operatorname{Re} \frac{f(p)}{p^s} \right)$ in the region $\operatorname{Re}(s) \ge 1$.
Then Lemma \ref{Halász} implies

\begin{equation} \label{exp(-sum f(p)/p sum f(n)}
\exp \left( - \sum_{p \le x} \frac{f(p)}{p} \right) \sum_{n \le x} f(n) \ll \frac{x}{\log x} H(x) + x,
\end{equation}
where 
\[
H(x) = 
\left( \sum_{|N| \le (\log x)^2 + 1} \frac{1}{N^2 + 1} \sup_{|t - N| \le 1/2} \exp \left( 2 \sum_{p \le x} \operatorname{Re} \frac{f(p)}{p} (p^{-it} - 1) \right) \right)^{1/2}.
\]

It is clear that $\operatorname{Re} \frac{f(p)}{p} \left( p^{- it} - 1 \right)$ is maximized when $f(p) = -1$. Hence

\begin{multline*}
H(x) \ll 
\left( \sum_{|N| \le (\log x)^2 + 1} \frac{1}{N^2 + 1} \sup_{|t - N| \le 1/2} \exp \left( 2 \sum_{p \le x} \operatorname{Re} \frac{(1 - p^{-it})}{p}  \right) \right)^{1/2} \ll
\\
(\log x) \left( \sum_{|N| \le (\log x)^2 + 1} \frac{1}{N^2 + 1} \sup_{|t - N| \le 1/2} \zeta^{-2}\left( 1 + \frac{1}{\log x} + i t \right) \right)^{1/2},
\end{multline*}

where we used Lemma \ref{sum_{p < x} to sum_p} (i), (ii).

Lemma \ref{sum_{p < x} to sum_p}(iii) implies that 
\[
 \sum_{|N| \le (\log x)^2 + 1} \frac{1}{N^2 + 1} \sup_{|t - N| \le 1/2} \zeta^{-2}\left( 1 + \frac{1}{\log x} + i t \right) = O(1).
\]
Thus $H(x) \ll \log x$. This finishes the proof in view of (\ref{exp(-sum f(p)/p sum f(n)}).
\end{proof}

\subsection{The end of the proof of Theorem \ref{Kerr and Klurman improvement}}

Let $y = x^{\varepsilon}$ and $\frac{\log_4 x}{\log_3 x} \ll \varepsilon < 1/2$.

Proposition \ref{sum f(n) upper bound in terms of sum f(p)/p} implies that for $y^{2} \le t \le y^3$ we have
\begin{equation} \label{c_f(y) with sum f(p)/p}
\sum_{n \in S(t, y)} f(n) \ll  \varepsilon^{-1} t \exp \left( \sum_{p \le x} \frac{f(p)}{p} \right).
\end{equation}

We see that the conditions of Proposition \ref{Psi_f proposition} are satisfied with $a = 2$ and
\[
c_f(y) = O \left(\varepsilon^{-1}  \exp \left( \sum_{p \le x} \frac{f(p)}{p}  \right) \right).
\]

Hence
\[
|\Psi_f(t, y)| \ll c_f(y) t \rho \left( \frac{\log t}{\log y} \right),
\]
for $y^2 \le t \le x$, if $\log y \ge (\log_2 x)^{5/3 + \varepsilon}$.

Let $\log x - h_1(x) < j \le \log x + 1$.
Integrating by parts, we obtain
\begin{multline} \label{sum f(l) {x/l} bound}
\sum_{\substack{l \in S(x, x^{\varepsilon}) \\ e^j \le l < e^{j + 1}}} f(l)  \left\{ \frac{x}{l} \right\} = 
\int_{e^{j}}^{e^{j + 1}} \left\{ \frac{x}{t} \right\} \, d\Psi_{f}(t, y) + O(1) \ll
\\
\frac{x}{e^{j}} \sup_{t \in [e^j, e^{j + 1}]} |\Psi_{f}(t, y)| + O(1) \ll c_f(y) x \rho \left(\varepsilon^{-1} - 
\varepsilon^{-1}\frac{h_1(x)}{\log x} \right) + O(1).
\end{multline}

We use (\ref{S_2 prefinal bound}), (\ref{sum f(l) {x/l} bound}) and the upper bound $\rho(u) \ll \exp( - (1 + o(1))u\log u )$ (see, for example, \cite[Corollary 2.3]{hildebrand1993integers}) 
to obtain
\begin{equation} \label{S_2 final bound}
S_2 \ll h_1(x) \varepsilon^{-1} 
\exp \left( - (1 + o(1)) \varepsilon^{-1} \log \varepsilon^{-1} \right).
\end{equation}

Now recall (\ref{S_1 final bound}). Let us choose $q = \varepsilon (\log x) x^{(1 - 2\delta) \varepsilon} + O(1)$, $\delta = (\log_3 x)^{-1}$, $h_1(x) = 10(\log_2 x) (\log_3 x)$. Finally let 
\[
\varepsilon = \frac{\log_4 x}{(1 + o(1)) \log_3 x},
\]
where $o(1)$ is chosen in such a way that $S_2 = o(1)$.

Hence for
\[
q = \exp \left( \frac{\log x \log_4 x}{(1 + o(1)) \log_3 x} \right),
\]
we have $S_1 = o(1), S_2 = o(1)$ and thus $S = o(1)$.

This finishes the proof of Theorem \ref{Kerr and Klurman improvement} in view of (\ref{P_x moment ineq}). \qed

\subsection{
Proof of Theorem \ref{cor from conj}
}

We do the same steps as in the proof of Theorem \ref{Kerr and Klurman improvement}, but we use  Conjecture \ref{smooth Halasz conj} instead of Proposition \ref{sum f(n) upper bound in terms of sum f(p)/p}.

In the notation of the proof of Theorem \ref{Kerr and Klurman improvement} this gives us 
\[
S_2 \ll \frac{h_1(x)  \varepsilon^{-1} 
\exp \left( - (1 + o(1)) \varepsilon^{-1} \log \varepsilon^{-1} \right)}{\log_2 x},
\]
where $\varepsilon$ is small enough.

Inequality (\ref{P_x moment ineq}) states that $P_x \le (c_1^{-1} S)^q + O(\exp(-x^{\beta}))$ for some fixed $c_1 > 0, \beta > 0$. 

Let us take $q = \varepsilon (\log x) x^{(1 - 2\delta) \varepsilon} + O(1)$, $\delta = 1/10$, $h_1(x) = 100 \log_2 x$, and $\varepsilon > 0$ to be a fixed constant such that $S_2 < c_1 / 3$. 

Our choice of variables implies that $S_1 = o(1)$, in view of (\ref{S_1 final bound}). Therefore
\[
P_x \ll \exp(-q) + \exp(-x^{\beta}) \ll \exp(-x^{\alpha}), 
\]
where $\alpha = \min(\beta, (8/10)\varepsilon)$. \qed

\subsection{
Proof of Proposition \ref{high moment sum f(n) ineq}}

Let 
\[
S' := \E \left[ \left( \sum_{n \le x} f(n) \right)^q \right]^{1/q}.
\]

Let $h_1(x) = o(\log x)$, $\varepsilon = o(\log x), \delta > 0$.
Following the steps of the proof of Theorem \ref{Kerr and Klurman improvement} (section \ref{steps like in Kerr and Klurman}), we obtain $S' \le S_1' + S_2'$, where
\[
S_1' \ll x \frac{\varepsilon}{\delta} (\log x) e^{-\delta h_1(x)} \exp \left( \frac{c_2 q}{\varepsilon (\log x) x^{(1 - 2 \delta) \varepsilon}} \right),
\quad
S_2' \ll h_1(x) \sup_{f} \left| \sum_{\substack{l \in S(x, x^{\varepsilon}) \\ e^{j_0} \le l < e^{j_0 + 1}}} f(l) \left\{ \frac{x}{l} \right\} \right|.
\]
Here $\log x - h_1(x) < j_0 \le \log x + 1$. Hence $S_2' \ll h_1(x) \Psi(x, x^{\varepsilon})$.

For $q \le \varepsilon (\log x) x^{(1 - 2 \delta) \varepsilon}$ this gives us
\begin{equation} \label{S' final bound}
S' \ll x \left( h_1(x) \exp(-(1 + o(1)) \varepsilon^{-1} \log \varepsilon^{-1}) \, + \, 
\frac{\varepsilon}{\delta} (\log x) e^{- \delta h_1(x)} \right).
\end{equation}

Let us choose $\delta = (\log_3 x)^{-1}$, $h_1(x) = (10 + \beta_0) (\log_2 x) (\log_3 x)$. Finally let
\[
\varepsilon = \frac{\log_3 x}{(\beta(x) + o(1))\log_2 x},
\]
where $o(1)$ is chosen in such a way that (\ref{S' final bound}) gives $S' = o\left( \frac{x}{(\log x)^{\beta(x)}} \right)$.

This finishes the proof of Proposition \ref{high moment sum f(n) ineq}. \qed

\section{Proof of Theorem \ref{chi(n)/n}}

\begin{lemma}[Elliott] \label{Second moment ineq}
 Let $b_1, b_2, \ldots$ be a sequence of complex numbers such that $\forall i \, |b_i| \le 1$. Then
\[
\E \left[ \left( \sum_{n \le y} b_n \chi_p(n) \right)^2 \right] \ll (\log x) y \log y + \frac{\log x}{x} y^3 \log y.
\]
\end{lemma}
\begin{proof}
    This follows from \cite[Lemma 10]{elliott1970mean} if we note that
    \[
    \sum_{\substack{m, n \le y \\ m n = \square}} |b_n b_m| \ll y \log y.
    \]
\end{proof}

\begin{lemma} \label{high moment inequality}
    Let $b_1, b_2, \ldots$ be a sequence of complex numbers such that $\forall i \, |b_i| \le 1$.
    Let $h(y)$ be a function such that $0 < h(y) \ll (\log y)^{o(1)}$ and let $q$ be an even positive integer. Then
    \[
    \E \left[ \left( \sum_{n \le y} b_n \chi_p(n) \right)^q \right] 
    \ll
    (\log x) o \left( \frac{y}{h(y) \log y} \right)^q +
    \frac{\log x}{x} (\log y^q) (4 y^{3/2})^q
    \]
    for all 
    \begin{equation} \label{q range}
    q \le \exp \left( \frac{\log y \log_3 y}{(1 + o_h(1)) \log_2 y} \right).
    \end{equation}
\end{lemma}

\begin{proof}
    We follow the proof of \cite[Lemma 10]{elliott1970mean}.
    We have
    \[
    \E \left[ \left( \sum_{n \le y} b_n \chi_p(n) \right)^q \right] \ll \frac{\log x}{x} \sum_{x < p \le 2x} 
    \left| \sum_{n \le y} b_n \left(\frac{n}{p}\right) \right|^{q}.
    \]
    We extend the definition of the Legendre symbol by
    \[
    \left( \frac{m}{2} \right)_* := 
    \begin{cases}
    1 \,\, \text{if} \,\, 2 \nmid m, \\
    0 \,\, \text{if} \,\, 2 \,|\, m,
    \end{cases}
    \]
    and
    \[
    \left( \frac{m}{n} \right)_* := \prod_{p^{\alpha} || n} \left( \frac{m}{p} \right)^{\alpha}.
    \]
    Note that this definition differs from the usual definition of the Kronecker symbol.

    Let $n = 2^{\eta} n_1, 2 \nmid n_1$. We divide all integers $n$ into four classes according to the parity of $\eta$, and whether $n_1 \equiv 1$ or $n_1 \equiv 3 \pmod{4}$. Let us denote by $\Sigma_j \,\, (j = 1, \ldots, 4)$ the summation over a particular class.

    The quadratic reciprocity law implies that for any odd integer $m$ we have
    \[
    \left( \frac{n}{m} \right)_*
    =
    \varepsilon \left( \frac{m}{n} \right)_*,
    \]
    where $\varepsilon = \pm 1$ and depends only on $m$ and the class of $n$.

    Hence
    \[
    \left| \sideset{}{_j}\sum_{n \le y} b_n \left( \frac{n}{m} \right)_* \right| = 
    \left| \sideset{}{_j}\sum_{n \le y} b_n \left( \frac{m}{n} \right)_* \right|.
    \]

    Jensen's inequality applied to the function $\varphi(z) = z^q$ implies that
    \[
    \sum_{x < p \le 2x} 
    \left| \sum_{n \le y} b_n \left(\frac{n}{p}\right) \right|^{q} \le 
    4^{q - 1} \sum_{j = 1}^4 
    \sum_{x < p \le 2x} 
    \left| \sideset{}{_j}\sum_{n \le y} b_n \left( \frac{n}{p} \right) \right|^{q}.
    \]

    It was shown in the proof of \cite[Lemma 10]{elliott1970mean} that
    \[
    m \mapsto \left( \frac{m}{n} \right)
    \]
    defines a non-principal character unless $n$ or $\frac{1}{2}n$ is a perfect square.

    Hence
    \begin{multline*}
    \sum_{x < p \le 2x} 
    \left| \sideset{}{_j}\sum_{n \le y} b_n \left( \frac{n}{p} \right) \right|^{q} =
    \sum_{x < p \le 2x} 
    \left| \sideset{}{_j}\sum_{n \le y} b_n \left( \frac{p}{n} \right) \right|^{q} 
    \le
    \sum_{x < m \le 2x} 
    \left| \sideset{}{_j}\sum_{n \le y} b_n \left( \frac{m}{n} \right)_* \right|^{q} 
    \ll
    \\
    x \sideset{}{_j}\sum_{\substack{n_1, \ldots, n_q \le y \\ n_1 \ldots n_q = \square, 2 \square}} b_{n_1} \ldots b_{n_{q/2}} \overline{b_{n_{q/2 + 1}} \ldots b_{n_q}} +
    \sideset{}{_j}\sum_{\substack{n_1, \ldots, n_q \le y \\ n_1 \ldots n_q \ne \square, 2 \square}} b_{n_1} \ldots b_{n_{q/2}} \overline{b_{n_{q/2 + 1}} \ldots b_{n_q}} \sum_{m \le x}\left( \frac{m}{n_1 \ldots n_q} \right)_*.
    \end{multline*}

Proposition \ref{high moment sum f(n) ineq} implies that
\[
\sum_{\substack{n_1, \ldots, n_q \le y \\ n_1 \ldots n_q = \square, 2 \square}} 1 \le (q + 1) 
\sum_{\substack{n_1, \ldots, n_q \le y \\ n_1 \ldots n_q = \square}} 1 = 
(q + 1) \E \left[ \left( \sum_{n \le y} f(n) \right)^q \right] = o \left( \frac{y}{h(y) \log y} \right)^q,
\]
for $q$ in range (\ref{q range}).

If $n_1 \ldots n_q \ne \square, 2 \square$, then the  P\'olya-Vinogradov inequality gives
\[
\sum_{m \le x}\left( \frac{m}{n_1 \ldots n_q} \right)_* \ll 
y^{q/2} \log(y^q).
\]

Combining all these estimates, we obtain the desired inequality after redefining
the function $h(y)$. 
\end{proof}

Let $M$ be a subset of $[1, \infty)$. Denote by $\tilde{P}_x(M)$ the probability that for any $y \in M$ we have $\sum_{n \le y} \frac{\chi_p(n)}{n} > 0$ and by $P(M)$ the probability that for any $y \in M$ we have $\sum_{n \le y} \frac{f(n)}{n} > 0$, where $f$ is a random completely multiplicative function.

\begin{lemma} \label{P_x([N, infty)]) bound}
For 
\[
\exp\left( \frac{\log_3 x \log_4 x}{o(1) \log_5 x} \right) \le N \le \exp \left( \frac{\log_2 x \log_3 x}{3 \log_4 x} \right)
\]
we have
\[
1 - \tilde{P}_x([N, \infty))  \ll \exp \left( - \exp \left( \frac{\log N \log_3 N}{(1 + o(1)) \log_{2} N} \right) \right).
\]
\end{lemma}

\begin{proof}
For each $p \in (x, 2x]$ the P\'olya-Vinogradov inequality implies that
\[
\sum_{n > y} \frac{\chi_p(n)}{n} \ll \frac{\sqrt{x} \log x}{y}.
\]

By \cite[Theorem 11.4]{montgomery2007multiplicative}, we have $L(1, \chi_p) \gg (\log x)^{-1}$ as long as $L(s, \chi_p)$ has no Siegel zero. The last condition holds for all $p \in (x, 2x]$ with at most one possible exception \cite[Corollary 11.9]{montgomery2007multiplicative}. The chosen prime is exceptional (has a Siegel zero) with probability $O(1/\log x)$, which is smaller than the error term in Lemma \ref{P_x([N, infty)]) bound}. Hence, from now on, we can assume that our random prime $p$ is nonexceptional and $L(1, \chi_p) \gg (\log x)^{-1}$.

Thus there exists $C > 0$ such that for any sufficiently large $x$, and $y \ge C x^{1/2} (\log x)^2$ we have
\[
\sum_{n \le y} \frac{\chi_p(n)}{n} = L(1, \chi_p) - \sum_{n > y} \frac{\chi_p(n)}{n} > 0.
\]

Let us denote $g_p = \chi_p * 1$.

We have 
\[
\sum_{n \le y} \frac{\chi_p(n)}{n} = \frac{1}{y} \sum_{n \le y} g_p(n) + \frac{1}{y} \sum_{n \le y} \chi_p(n) \left\{ \frac{y}{n} \right\}. 
\]

Let us denote
\[
Pr_1(y) = \Prob\left( \frac{1}{y} \left| \sum_{n \le y} \chi_p(n) \left\{ \frac{y}{n} \right\} \right| > \frac{0.1}{\log y} \right), 
\quad 
Pr_2(y) = \Prob\left( \frac{1}{y} \sum_{n \le y} g_p(n) < \frac{0.2}{\log y} \right).
\]

Note that if both events do not take place, then $\sum_{n \le y} \frac{\chi_p(n)}{n} > \frac{0.1}{\log y}$. Thus the inequality $\sum_{n \le y'} \frac{\chi_p(n)}{n} > 0$ holds for all $y' \in \left[y, y + \frac{y}{10^2 \log y} \right]$.

Let $y_0 = N$ and $y_{i + 1} = y_i + \frac{y_i}{10^2 \log y_{i}}$. Suppose that $k$ is the least number such that $y_k > C x^{1/2} (\log x)^2$. Note that $k \ll (\log x)^2$. 

It is clear that
\[
1 - \tilde{P}_x([N, \infty)) \le \sum_{i = 0}^k \left( Pr_1(y_i) + Pr_2(y_i) \right).
\]

Hence, it is enough to prove that
\[
\sum_{i = 0}^k \left( Pr_1(y_i) + Pr_2(y_i) \right) \ll \exp \left( - \exp \left( \frac{\log N \log_3 N}{(1 + o(1)) \log_{2} N} \right) \right).
\]

By assumption, $N \le \exp \left( \frac{\log_2 x \log_3 x}{3 \log_4 x} \right)$, and thus
\[
\exp \left( - \exp \left( \frac{\log N \log_3 N}{(1 + o(1)) \log_{2} N} \right) \right) \gg B(x) := \exp \left( - (\log x)^{\frac{1}{3} + o(1)} \right).
\]

Since
\[
\sum_{n \le y} g_p(n) \ge \sum_{q \le y} g_p(q) = \sum_{q \le y} (1 + \chi_p(q)),
\]
where $q$ ranges over prime numbers, we get
\[
Pr_2(y) \le \Prob \left( \frac{1}{y} \left|  \sum_{q \le y} \chi_p(q) \right| > \frac{0.7}{\log y}  \right).
\]

For $y_i > \exp(\sqrt{\log x})$ we apply Lemma \ref{Second moment ineq} and Markov's inequality to obtain
\[
\left( Pr_1(y_i) + Pr_2(y_i) \right) \ll (\log x)(\log y_i)^3 \left( \frac{1}{y_i} + \frac{y_i}{x} \right) \ll \exp\left( -(\log x)^{\frac{1}{2} + o(1)} \right).
\]
As $k \ll (\log x)^2$, we have
\[
\sum_{\substack{y_i > \exp(\sqrt{\log x}) \\ i \le k}} \left( Pr_1(y_i) + Pr_2(y_i) \right) \ll \exp\left( -(\log x)^{\frac{1}{2} + o(1)} \right) \ll B(x).
\]

Now for $y_i \le \exp(\sqrt{\log x})$ we apply Lemma \ref{high moment inequality} with $h(y) = 10$. We take $q$ as large as possible with the restrictions
\[
q \le D(y_i) := \exp\left( \frac{\log y_i \log_3 y_i}{(1 + o_h(1)) \log_2 y_i} \right), \quad q \le E(y_i) := \frac{\log x}{10 \log (4 y_i^{1/2} \log y_i)}.
\]
The last restriction in view of Lemma \ref{high moment inequality} and Markov's inequality implies that
\[
\left( Pr_1(y_i) + Pr_2(y_i) \right) \ll (\log x) \exp(-q) + O(x^{-1/2}).
\]

Therefore
\begin{multline*}
\sum_{y_i \le \exp(\sqrt{\log x})} \left( Pr_1(y_i) + Pr_2(y_i) \right) \ll \\
 \sum_{y_i \le \exp(\sqrt{\log x})} \left( (\log x) \left( \exp(-D(y_i)) + \exp(-E(y_i)) \right) + O(x^{-1/2})\right) \ll \\
(\log x)  \sum_{y_i \le \exp(\sqrt{\log x})}\left(\exp(-D(y_i)) + \exp(-E(y_i)) \right) + O(x^{-1/2} (\log x)^2).
\end{multline*}

We have
\[
(\log x) \sum_{y_i \le \exp(\sqrt{\log x})} \exp(-E(y_i)) \ll (\log x)^3 \exp\left( - \sqrt{\log x} \right) \ll B(x).
\]

Also
\begin{multline*}
(\log x) \sum_{y_i \le \exp(\sqrt{\log x})} \exp(-D(y_i)) \ll
\\
(\log x)^3  \exp \left( - \exp \left( \frac{\log N \log_3 N}{(1 + o(1)) \log_{2} N} \right) \right) \ll  \exp \left( - \exp \left( \frac{\log N \log_3 N}{(1 + o(1)) \log_{2} N} \right) \right).
\end{multline*}
Here we used the assumed lower bound on $N$.

Combining all estimates, the result follows.
\end{proof}

Let 
\[
\pi(x; k, l) := \sum_{\substack{p \le x \\ p \equiv l \pmod{k}}} 1.
\]

\begin{lemma} \label{pi(x; k, l)}
Let $k \le \exp(C \sqrt{\log x})$.
\[
\pi(x; k, l) = \frac{\operatorname{Li}(x)}{\varphi(k)} - E_1 \frac{\chi_1(l)}{\varphi(k)} \int_2^x \frac{u^{\beta_1 - 1}}{\log u} \, du + O \left( x \exp(-c' \sqrt{\log x}) \right),
\]
where $E_1 = 1$ if there exists a quadratic Dirichlet character $\chi_1 \pmod k$ with real zero $\beta_1$ such that $\beta_1 > 1 - \frac{c}{\log k}$ and $E_1 = 0$ otherwise.

Moreover, if such character exists, then it is unique and $\chi_1(l) = \left( \frac{d}{l} \right)$, where $d$ is the product of relatively prime factors of the form
\[
-4, \,\, 8, \,\, -8, \,\, (-1)^{(p-1)/2} p \quad (p > 2).
\]
Also $|d|$ is the conductor of $\chi_1$ and $d \log^4 d \gg \log x$.
\end{lemma}

\begin{proof}
The first part of the Lemma follows from \cite[Chapter 20, equation 9]{davenport2013multiplicative} after integration by parts. For the second part see \cite[Chapter 5, equation 9]{davenport2013multiplicative} and \cite[Chapter 20, equation 12]{davenport2013multiplicative}.
\end{proof}

\begin{proof}[Proof of Theorem \ref{chi(n)/n}]
Let us take $N = c_1 \sqrt{\log x}$ and $k = 8 \prod_{2 < q \le N} q$, where $c_1$ is sufficiently small. Clearly $k \ll \exp(c_2 \sqrt{\log x})$.

Now we want to compare $\tilde{P}_x([1, N])$ and $P([1, N])$.
Let $f(n)$ be a sample of a Rademacher random completely multiplicative function. Also let us set $\Prob(f(-1) = 1) = \Prob(f(-1) = -1) = 1/2$. 

Let us denote
\[
S(f) = \left\{ l \pmod{k} : (l, k) = 1, \, p \equiv l \pmod{k} \Rightarrow \forall q \le N \, \left( \frac{q}{p} \right) = f(q), \left( \frac{-1}{p}\right)  = f(-1) \right\}.
\]

The quadratic reciprocity law implies that 
\[
\left| S(f) \right| = \frac{\varphi(k)}{2^{\pi(N) + 1}}.
\]

Also we note that for each $l \in S(f)$ we have
\[
\chi_1(l) = \prod_{q^{\alpha_q} || d} \left( \frac{q}{l} \right)^{\alpha_q} = f(d),
\]
where $q = -1, \alpha_{-1} = 1$ is included in the product if $d < 0$.

Putting all this together we obtain for a fixed $f$ 
\begin{equation} \label{chi_p = f probability}
\Prob \left( \forall n \in [1, N]  \,\, \chi_p(n) = f(n) \right) =
\frac{1}{2^{\pi(N)}} - E_0 \frac{f(d)}{2^{\pi(N)}} \frac{\int_x^{2x} \frac{u^{\beta_1 - 1}}{\log u} \, du}{\operatorname{Li}(2x) - \operatorname{Li}(x)} + O \left(\exp(-c' \sqrt{\log x}) \right),
\end{equation}
where $E_0 = 0$ if $E_1 = 0$ or $d < 0$, and $E_0 = 1$ otherwise.

Let $A_N$ be the set of completely multiplicative functions $f$ defined on $[1, N]$ that take values $\pm 1$ and such that $\sum_{n \le y} \frac{f(n)}{n}$ is positive for $1 \le y \le N$.

From (\ref{chi_p = f probability}) we deduce that
\begin{equation} \label{eq01}
\tilde{P}_x([1, N]) = P([1, N]) - E_0  \frac{\sum_{f \in A_N}f(d)}
{2^{\pi(N)}} \frac{\int_x^{2x} \frac{u^{\beta_1 - 1}}{\log u} \, du}{\operatorname{Li}(2x) - \operatorname{Li}(x)} + O(\exp(-c'' \sqrt{\log x})).
\end{equation}

We have
\begin{equation} \label{eq02}
\frac{\sum_{f \in A_N}f(d)}
{2^{\pi(N)}} = \Cov \left( \mathds{1}_{A_N}, f(d) \right).
\end{equation}
The right-hand side should be interpreted as the covariance in $\mathcal{F}$.

Also
\begin{equation} \label{eq03}
\left| \Cov \left( \mathds{1}_{A_N}, f(d) \right) - \Cov \left( \mathds{1}_{A}, f(d) \right)\right| \le 1 - P((N, \infty)).
\end{equation}
Finally, by Lemma \ref{P_x([N, infty)]) bound} and Theorem \ref{Kerr and Klurman improvement} (or \cite[Theorem 1.2]{kerr2022negative}), we obtain
\begin{equation} \label{eq04}
(1 - P((N, \infty))) + (1 - \tilde{P}_x((N, \infty))) \ll \exp \left( - \exp \left( \frac{\log_2 x \log_4 x}{(2 + o(1)) \log_3 x} \right) \right).
\end{equation}

Since 
\[
\tilde{P}_x = \tilde{P}_x([1, N]) + O(1 - \tilde{P}_x((N, \infty))), \quad P = P([1, N]) + O(1 - P((N, \infty))),
\]
the theorem follows from (\ref{eq01}), (\ref{eq02}), (\ref{eq03}), (\ref{eq04}).
\end{proof}

\subsection{Proof of Corollary \ref{Cor chi(n)/n}}

If $E_0 = 0$, then the result is obvious. Assume that $E_0 = 1$ and $d$ is the conductor of the character with Siegel zero.

Denote by $p_0$ the greatest prime divisor of $d$. Lemma \ref{pi(x; k, l)} implies that $p_0 \ge (1 + o(1)) \log d \ge (1 + o(1)) \log_2(x)$. Take $N' = p_0 - 1$.
Note that 
\[
\Cov \left( \mathds{1}_{A_{N'}}, f(d) \right) = 0,
\]
and
\begin{multline*}
\Cov \left( \mathds{1}_{A}, f(d) \right) - \Cov \left( \mathds{1}_{A_{N'}}, f(d) \right) \ll 1 - P((N', \infty)) \ll \\
\exp \left( - \exp \left( \frac{\log_3 x \log_6 x}{(1 + o(1)) \log_5 x} \right) \right),
\end{multline*}
where we used Theorem \ref{Kerr and Klurman improvement}.
The result follows. \qed

\textbf{Acknowledgements.}
This work was supported by the Russian Science Foundation under grant no 24-71-10005, \url{https://rscf.ru/project/24-71-10005/}.
I am grateful to Oleksiy Klurman and my advisor A. B. Kalmynin for valuable discussions.

\bibliography{mult}
\bibliographystyle{plainurl}

\end{document}